\documentclass{amsart}
\usepackage{amssymb}
\usepackage[hyphens]{url} 
\usepackage{enumerate}
\usepackage{amsmath,amsfonts,amsthm}
\usepackage{cases}
\usepackage{amscd,mathtools}
\usepackage{tagging}
\usepackage{graphicx}
\usepackage{subcaption}
\usepackage{float}
\allowdisplaybreaks
\usepackage{nameref}
\usepackage[numeric]{amsrefs}
\usepackage{hyperref}
 \newtheorem{thm}{Theorem}[section]
 \newtheorem{prop}[thm]{Proposition}
 \newtheorem{lem}[thm]{Lemma}
 \newtheorem{cor}[thm]{Corollary}
\theoremstyle{definition}
 \newtheorem{exm}[thm]{Example}
 \newtheorem{defn}[thm]{Definition}
 \newtheorem{rem}[thm]{Remark}
 \numberwithin{equation}{section}
 \numberwithin{figure}{section}

\renewcommand{\le}{\leqslant}
\renewcommand{\ge}{\geqslant}

\title[Prolongation of solutions and Lyapunov-like stability results]{Prolongation of solutions and Lyapunov stability for Stieltjes dynamical systems}

\author{Lamiae Maia}
\address{Lamiae Maia
 \newline
  LAMA Laboratory,
   \newline
  Mathematics Department,
 \newline
 Faculty of Sciences,
 \newline
  Mohammed V University in Rabat,
 \newline
  Rabat,
  \newline
   Morocco.
  }
\email{lamiae\_maia@um5.ac.ma}
\author{Noha El Khattabi}
\address{Noha El Khattabi
 \newline
  LAMA Laboratory,
   \newline
  Mathematics Department,
 \newline
 Faculty of Sciences,
 \newline
  Mohammed V University in Rabat,
 \newline
  Rabat,
  \newline
   Morocco.
  }
\email{noha.elkhattabi@fsr.um5.ac.ma}
\author{Marl\`ene Frigon}
\address{Marl\`ene Frigon
 \newline
Department of Mathematics and Statistics,
 \newline
  University of Montreal,
 \newline
   Montreal, H3C 3J7,
  \newline
  Canada.}
\email{marlene.frigon@umontreal.ca}



\begin{document}

\begin{abstract}
In this article, we introduce Lyapunov-type results to investigate the stability of the trivial solution of a Stieltjes dynamical system. We utilize prolongation results to establish the global existence of the maximal solution. Using Lyapunov's second method, we establish results of (uniform) stability and (uniform) asymptotic stability by employing a Lyapunov function. Additionally, we present examples and real-life applications to study asymptotic stability of equilibria in two population dynamics models.
\end{abstract}

\maketitle

\section{Introduction}
In recent years, the field of differential equations has witnessed a notable surge in interest surrounding Stieltjes differential equations. This renewed focus is largely driven by the pursuit of results that not only unify existing findings but also extend those related to classical derivatives~\cite{FerTojoVill,FP,PM,PM2,MEF1,Mar,MM,SS,SS-2} through the  Stieltjes derivative.

The distinction between classical derivatives and their Stieltjes counterparts lies in the nature of the differentiation process. While classical derivatives are based on limits involving small increments of function values, the Stieltjes derivative operates with respect to a derivator $g:\mathbb{R}\to \mathbb{R}$. This derivator is typically assumed to be left-continuous and nondecreasing, characteristics that allow for a broader range of applications, particularly in contexts where certain processes may exhibit discontinuities and/or stationary periods~\cite{AFNT,FerTojoVill,FP,PM,PM2,PM3,PMM,MEF1,MEF2}. Such scenarios are common in various fields, including population dynamics, and physics, where classical differentiation has limitations in capturing the complexities of real-world phenomena.

Typically, investigations into first-order Stieltjes differential equations and systems focus on solutions defined on bounded intervals. However, Larivière in~\cite[Chapter~4]{ThesisFL} turned his attention to Stieltjes differential equations on the positive real half-line. In doing so, he provided results related to the prolongation of solutions and the existence of the maximal solution. The motivation behind exploring these equations on the positive real half-line lies in the observation that many natural processes evolve over time without any inherent time limit, while some phenomena can exhibit finite-time blow-up, leading to abrupt changes or singularities. By considering the Stieltjes derivative, we aim to capture these nuanced behaviors and enhance our understanding of more complex dynamical systems from a common perspective, facilitated by the unification that this derivative provides~\cite{PR}.

In the study of dynamical systems, the stability of {\it equilibria} holds significant importance. Here, the term {\it equilibrium} refers to a state that does not change dynamically, in the sense that if a system starts at an equilibrium point, it will stay in that state indefinitely. The interaction between species within many ecosystems often relies on feedback loops, which makes stability crucial for their functioning, and resilience. Although stability is typically desirable, there are some scenarios where stable equilibrium at zero can be critical and raise concerns for several raisons. For instance, in ecological systems, this concern arises from the vulnerability to perturbations of certain species, which may increase the risk of their extinction.

Within the realm of dynamical systems theory, Lyapunov's Second Method~\cite{L1} stands as a fundamental approach for assessing the stability properties of a system near an equilibrium.  This stability analysis provides insights into whether small perturbations around an equilibrium lead to convergence (stability) or divergence (instability) of solutions. The core of this method lies in the concept of the {\it Lyapunov function}. This function was the subject of numerous works in the classical literature starting from the works~\cite{H1,H2,Kaymakcalan,Ko,L1,Yoshizawa} and references therein.

In this paper, we extend  Lyapunov stability results from the classical literature to the Stieltjes case by means of the Stieltjes derivative. In doing so, we address first the prolongation of solutions considering the Stieltjes dynamical system:
\begin{equation}\label{eq:dyn-system}
\begin{aligned}
     \mathbf{x}_g'(t) &= \mathbf{f}(t,\mathbf{x}(t)) \quad \text{for $g$-almost all $t\ge t_0\ge 0$,}
     \\
      \mathbf{x}(t_0) &= \mathbf{x}_0\in\mathbb{R}^n;
\end{aligned}
\end{equation}
where $ \mathbf{f}=(f_1,\dots,f_n):[0,+\infty) \times \mathbb{R}^n \to \mathbb{R}^n$. We start by deriving corollary results using compact sets to characterize the maximal solution of~\eqref{eq:dyn-system}. Then, based on a generalized version of the Gr\"{o}nwall lemma~\cite{ThesisFL,GallMarSlav2025} for the Stieltjes derivative, we establish the existence of global solutions over the whole positive real half-line. The prolongation of solutions will be essentially used later on to establish Lyapunov-like stability results for the system~\eqref{eq:dyn-system}, particulary when studying the asymptotic behaviour of solutions around an equilibrium. Our stability study is inspired by results from classical theory and works such as~\cite{HT,Ko,Y2001,YLXD}, which address dynamic equations on time scales and impulsive differential equations. To the best of our knowledge, this is the first work to introduce Lyapunov's method adapted to Stieltjes differential equations.

This  paper is organized as follows: we present the theoretical framework and some preliminaries in Section~2.  In Section~3, we focus on prolongation of solutions and the characterization of the maximal solution to deduce the existence of a global solution. Section~4 is devoted to Lyapunov-like stability results using Lyapunov's second method. We start by defining the stability notions: (uniform) stability and (uniform) asymptotic stability. Afterward, we present stability results for each type of stability inspired from the works~\cite{FedGraMesToon,GallGraMes}, to extend some classical results from~\cite{HT,Ko}. In the last section of this paper, we suggest two applications to dynamics of population to study the asymptotic stability of some critical equilibria: in the first application, we  model the dynamics of a population with Allee's effect~\cite{BereAnCour,SteSuth} negatively impacted by train vibrations, and another application related to the dynamics of a population of {\it Cyanobacteria} in a cultured environment, keeping track of ammonia levels in the process.
\section{Preliminaries}
For $[a,b]\subset\mathbb{R}$, and $\mathbf{u}:[a,b]\to\mathbb{R}^n$ a regulated function,  the symbols $\mathbf{u}(t^+)$ and $\mathbf{u}(t^-)$  will be used to denote
\begin{align*}
\mathbf{u}(t^+)=\lim\limits_{\epsilon\to 0^+} \mathbf{u}(t+\epsilon), & \text{ for all }t\in[a,b), \\
\mathbf{u}(t^-)=\lim\limits_{\epsilon\to 0^-} \mathbf{u}(t-\epsilon), & \text{ for all }t\in (a,b].
\end{align*}

Throughout this work, we will consider $g:\mathbb{R}\to \mathbb{R}$ a monotone, nondecreasing and left-continuous function, also known as a {\it derivator}. We denote the set of discontinuity points of~$g$ by
$$
D_g=\{t\in \mathbb{R}: g(t^+ )- g(t)>0\}.
$$
In addition, we denote
$$
C_g =\{s\in \mathbb{R}: g \text{ is constant on } (s-\epsilon, s+\epsilon) \text{ for some } \epsilon>0\}.
$$
The set $C_g$ is an open of the usual topology and it can be written as a countable union of disjoint intervals
$$C_g =\bigcup_{n\in \Lambda}(a_n ,b_n),$$
with $\Lambda \subset\mathbb{N}$, $a_n,b_n \in\mathbb{R}$. We set $N_g=\{a_n,b_n: n\in \Lambda\}\setminus D_g$.

The function $g$ defines a Lebesgue-Stieltjes measure $\mu_{g}:\mathcal{LS}_g\to [0,\infty]$ over $\mathcal{LS}_g$ the $\sigma$-algebra of subsets of $\mathbb{R}$ containing all Borel sets~\cite{AL,PR,Rudin}. We refer to the measurability with respect to the $\sigma$-algebra $\mathcal{LS}_g$ as {\it$g$-measurability}. For any interval $[a,b) \in\mathcal{LS}_g$, we have $\mu_g([a,b))=g(b)-g(a)$, and $\mu_g(\{t\})=g(t^+)-g(t)$ for all $t\in \mathbb{R}$. Moreover, $\mu_g(C_g)=\mu_g(N_g)=0$, the reader is referred to~\cite{PR} for more details.
We denote by  $\mathcal{L}_{g}^{1}([a,b),\mathbb{R})$ the quotient space of the set of $\mu_g$-integrable functions on $[a,b)\subset\mathbb{R}$ under the equivalence relation:
 $$
 f\sim h :\iff f(t)=h(t)\quad \mu_g\text{-almost everywhere}.
 $$
We define the norm $\|\cdot\|_{\mathcal{L}_{g}^{1}([a,b),\mathbb{R})}$ on the space $\mathcal{L}_{g}^{1}([a,b),\mathbb{R})$ by
$$
\|f\|_{\mathcal{L}_{g}^{1}([a,b),\mathbb{R})}:=\int_{[a,b)} |f(t)|\, d\mu_g(t), \quad \text{ for every } f\in \mathcal{L}_{g}^{1}([a,b),\mathbb{R}).
$$
In the sequel, given $I\subset\mathbb{R}$ and a functional space $Z(I,\mathbb{R})$, we set $Z(I,\mathbb{R}^n):=\prod_{i=1}^{n}Z(I,\mathbb{R})$, and we denote $Z_{loc}(I,\mathbb{R}^n)$ the space of functions $\mathbf{u}:I \to \mathbb{R}^n$ satisfying $\mathbf{u}|_{[a,b]}\in Z([a,b],\mathbb{R}^n)$ for every interval $[a,b] \subset I$.

The derivator $g$ defines a pseudometric $\rho:\mathbb{R}\times \mathbb{R} \to \mathbb{R}^+$ given by
$$
\rho(s,t)=|g(s)-g(t)|, \quad \text{ for every } s,t \in \mathbb{R}.
$$
We denote $\tau_g$ the topology induced by the pseudometric $\rho$. The topology $\tau_g$ is not necessarily Hausdorff as mentioned in~\cite[Section~2]{FP}.
\begin{defn} A set $A \subset \mathbb{R}$ is called {\it $g$-open\/} if, for every $t \in A$, there exists $r > 0$ such that
$$
\{s \in \mathbb{R} : \rho(t,s)< r\} \subset A.
 $$
 \end{defn}
If $t \in D_g$, then the interval $(t-\epsilon,t]$ is a $g$-open set. The reader is referred to~\cite[Section~2]{FP} for more properties of the topology $\tau_g$.

In the following definition, we define the $g$-derivative of a real-valued function.
\begin{defn}\label{df:g derivative}
Let $u:[a,b] \to \mathbb{R}$ be a function. The derivative of $u$ with respect to $g$ at a point $t_0 \in [a,b]\setminus C_g$, is defined as:
$$
u_g'(t)=
\begin{dcases}
\lim\limits_{s\to t} \frac{u(s)-u(t)}{g(s)-g(t)} & \mbox{if }  t \notin D_g,\\
\frac{u(t^+)-u(t)}{g(t^+)-g(t)} & \mbox{if } t \in D_g,
\end{dcases}
$$
provided that the limit exists, in this case $u$ is said to be $g$-differentiable at $t$.
\end{defn}

In the next proposition, we appeal to the $g$-derivative of the composition of two functions  established in~\cite[Proposition~3.15]{MarThesis}. Another version of this formula can be found in~\cite[Proposition~4.1]{FMarTo-OnFirstandSec}.
\begin{prop}\label{prop:g derivative of composition}
Let $t\in\mathbb{R}\setminus C_g$, $f:\mathbb{R}\to \mathbb{R}$ and $h$ a real function defined on a neighborhood of $f(t)$. We assume that there exist $h'(f(t))$, $f_g'(t)$ and that the function $h$ is continuous at $f(t^+)$.
Then, the composition $h\circ f$ is $g$-differentiable in $t_0$, and we have the formulae
$$
(h\circ f)_g'(t)=
\begin{dcases}
h'(f(t))f_g'(t) & \mbox{if }  t \notin D_g,  \\
\frac{h(f(t^+))-h(f(t))}{f(t^+)-f(t)}f_g'(t) & \mbox{if } t \in D_g.
\end{dcases}
$$
\end{prop}

Throughout this paper, let $\|\cdot\|$ denotes the maximum norm in $\mathbb{R}^n$ defined by
$$
\|\mathbf{x}\|=\max\{|x_1|,\dots,|x_n|\} \quad  \text{ for } \mathbf{x}=(x_1,\dots,x_n)\in\mathbb{R}^n.
$$

Now, we recall the $g$-continuity notion  first introduced in~\cite{FP}.
\begin{defn}\label{df:g-cont}
Let $\mathbf{u} : [a,b]\to \mathbb{R}^n$. We say that $\mathbf{u}$ is {\it $g$-continuous} at $t\in [a,b]$ if, for every $\epsilon >0$, there exists $\delta>0$ such that
    $$
    \forall s\in [a,b], \qquad |g(s)-g(t)|< \delta \implies \|\mathbf{u}(s)-\mathbf{u}(t)\|<\epsilon.
    $$
\end{defn}
We denote $\mathcal{C}_g([a,b],\mathbb{R}^n)$ the set of $g$-continuous functions $\mathbf{u}:[a,b]\to \mathbb{R}^n$ on the interval $[a,b]$. The following proposition relates the regularity of $f$ and $g$, the reader is referred to~\cite[Proposition~3.2]{FP}.
\begin{prop}\label{prop:g-conti funct constant where g constant}
  If $\mathbf{u}:[a,b] \to\mathbb{R}^n$ is $g$-continuous on $[a,b]$, then the following statements hold:
\begin{enumerate}[$(1)$]
    \item $\mathbf{u}$ is left-continuous at every $t \in (a,b]$.
    \item If $g$ is continuous at $t \in [a,b)$, then so is $\mathbf{u}$.
    \item If $g$ is constant on some $[c,d] \subset [a,b]$, then so is $\mathbf{u}$.
\end{enumerate}
\end{prop}
$g$-continuous functions presents interesting measurability properties, see~\cite[Corollary~3.5]{FP}. Let $\mathcal{BC}_g([a,b],\mathbb{R}^n)$ denotes the Banach space of the bounded functions of $\mathcal{C}_g([a,b],\mathbb{R}^n)$  with respect to the supremum norm.

In the next theorem, we focus on the particular case when the derivator $g:\mathbb{R}\to\mathbb{R}$ is increasing, and continuous on an interval $[a,b]\subset \mathbb{R}$, to derive a generalized version of Rolle's theorem and the Mean Value theorem for real-valued functions in the context of Stieltjes differentiation.

\begin{thm}\label{thm:g-Rolle theorem}
Let $g:\mathbb{R}\to\mathbb{R}$ be a left-continuous and nondecreasing function, continuous and increasing on an interval $[a,b]\subset \mathbb{R}$. Let $f:[a,b]\to \mathbb{R}$ be $g$-continuous on $[a,b]$ and $g$-differentiable on $(a,b)$ satisfying $f(a)=f(b)$. Then, there exists $c\in (a,b)$ such that $f_g'(c)=0$.
\end{thm}
\begin{proof}
 As $f$ is $g$-continuous on $[a,b]$, it results from Proposition~\ref{prop:g-conti funct constant where g constant} that $f$ is continuous on $[a,b]$.  We set
$$
m:=\min_{t\in[a,b]}f(t),\quad M:=\max_{t\in[a,b]}f(t).
$$
If $m=M$, then $f$ is constant on $[a,b]$ and for all $t\in (a,b)$, $f_g'(t)=0$. Otherwise if $m<M$, as $f(a)=f(b)$ we have
$$
f(a)\neq m,\text{ and } f(b)\neq M.
$$
Without loss of generality, we can assume that
$$
f(a)\neq M,\text{ and } f(b)\neq M,
$$
if it is not the case, we consider $m$ instead. Thus, there exists $c\in (a,b)$ such that $f(c)=M$. Therefore, there exists $\delta>0$ such that $(c-\delta,c+\delta)\subset(a,b)$ and $f(s)\le f(c)=M$ for all $s\in(c-\delta,c+\delta)$. As $s\nearrow c$, $g(s)\le g(c)$, and
$$
\lim\limits_{s\to c^-} \frac{f(s)-f(c)}{g(s)- g(c)}\ge 0.
$$
While, as $s\searrow c$, $g(s)\ge g(c)$, and
$$
\lim\limits_{s\to c^+} \frac{f(s)-f(c)}{g(s)- g(c)}\le 0.
$$
Since $f$ is $g$-differentiable at $c$, we deduce that $f_g'(c)=0$.
\end{proof}
As a corollary of Theorem~\ref{thm:g-Rolle theorem}, we state a version of the Mean Value theorem involving the Stieltjes derivative.
\begin{cor}\label{cor:g-MeanValueTheorem}
Let $g:\mathbb{R}\to\mathbb{R}$ be a left-continuous and nondecreasing function, continuous and increasing on an interval $[a,b]\subset \mathbb{R}$. Let $f:[a,b]\to \mathbb{R}$ be $g$-continuous on $[a,b]$ and $g$-differentiable on $(a,b)$. Then, there exists $c\in (a,b)$ such that $f_g'(c)=\frac{f(b)-f(a)}{g(b)-g(a)}$.
\end{cor}
\begin{proof}
  Let us consider the function $F:[a,b]\to \mathbb{R}$ defined by
  $$
  F(t)=f(t)-\left(\frac{f(b)-f(a)}{g(b)-g(a)}(g(t)-g(a))+f(a)\right),\quad\text{ for all }t\in [a,b].
  $$
Clearly $F$ is $g$-continuous on $[a,b]$ and $g$-differentiable on $(a,b)$, satisfying $F(a)=F(b)$.  Moreover, for all $t\in (a,b)$, we have that
$$
F_g'(t)=f_g'(t)-\frac{f(b)-f(a)}{g(b)-g(a)}.
$$
Applying Theorem~\ref{thm:g-Rolle theorem}, there exists $c\in(a,b)$ such that $F_g'(c)=0$. Hence, there exists $c\in(a,b)$ such that $f_g'(c)=\frac{f(b)-f(a)}{g(b)-g(a)}$.
\end{proof}

Now, we present the notion of $g$-absolute continuity.
\begin{defn}\label{df:g-abs cont}
A map $F:[a,b]\to \mathbb{R}$ is {\it $g$-absolutely continuous}, if, for every $\varepsilon > 0$, there exists $\delta > 0$ such that, for any family
$\{(a_i , b_i)\}_{i=1}^{i=n}$ of pairwise disjoint open subintervals of $[a,b]$,
$$
\sum_{i=1}^{n} g(b_i)-g(a_i) <\delta \Rightarrow \sum_{i=1}^{n} |F(b_i)-F(a_i)| < \varepsilon.
$$
\end{defn}

We denote by $\mathcal{AC}_g ([a,b],\mathbb{R})$ the vector space of $g$-absolutely continuous functions $F: [a,b]\to \mathbb R$ on the interval $[a,b]$.
In~\cite[Theorem~5.4]{PR}, a Fundamental Theorem of Calculus for Lebesgue-Stieltjes integrals was introduced.
\begin{thm}[Fundamental Theorem of Calculus for the Lebesgue-Stieltjes integral]\label{th:Fund-Th}
Let $a,b \in \mathbb{R}$ be such that $a<b$, and let $F:[a,b] \to \mathbb{R}$. The following assumptions are equivalent.
\begin{enumerate}
\item[{\rm(1)}] The function $F$ is {\em $g$-absolutely continuous}, i.e.  to each $\varepsilon>0$, there is some $\delta>0$ such that, for any family $\{(a_j,b_j)\}_{j=1}^{m}$ of pairwise disjoint open subintervals of $[a,b]$,
$$
\sum_{j=1}^{m}(g(b_j)-g(a_j))<\delta \quad
\implies \quad
\sum_{j=1}^m|F(b_j)-F(a_j)|<\varepsilon.
$$
\item[{\rm(2)}] The function $F$ satisfies the following conditions:
\begin{enumerate}
\item[{\rm(a)}] there exists $F'_g(t)$ for $g$-almost all $t\in [a,b)$;
\item[{\rm(b)}] $F'_g \in \mathcal{L}^1_g([a,b),\mathbb{R})$;
\item[{\rm(c)}] for each $t  \in [a,b]$, we have
$$
F(t)=F(a)+\int_{[a,t)} F'_g(s) \, d\mu_g(s).
$$
\end{enumerate}
\end{enumerate}
\end{thm}
We denote by $\mathcal{AC}_g([a,b],\mathbb{R})$ the set of  $g$-absolutely continuous functions. We set $\mathbf{u}=(u_1,\dots,u_n)\in\mathcal{AC}_g([a,b],\mathbb{R}^n)$ if  $u_i\in \mathcal{AC}_g([a,b],\mathbb{R})$ for $i=1,\dots,n$.

The following proposition provides conditions ensuring the $g$-absolute continuity of the composition of two functions, the proof is based on arguments as in~\cite[Proposition~5.3]{FP}.
\begin{lem}\label{lem:Lipschitz o g-absolute func}
  Let $\mathbf{u}:[a,b]\to B\subset\mathbb{R}^n$ be a $g$-absolutely continuous function, and let $v:B\to\mathbb{R}$ be a Lipschitz continuous function on $B$. Then, the composition~$v\circ\mathbf{u} \in \mathcal{AC}_{g}([a,b],\mathbb{R})$.
\end{lem}

In~\cite[Definition~6.1]{FP}, an exponential function was introduced.

\begin{defn}\label{df:g-exp}
Let $p \in \mathcal{L}^1_g([a,b),\mathbb{R})$ be such that
\begin{equation}\label{eq:exp:nonreasonance Condition}
1+p(t)\big(g(t^+)-g(t)\big) > 0\text{ for every } t \in [a,b)\cap D_g.
\end{equation}
Let us define the function $e_p(\cdot,a) : [a,b] \to (0,\infty)$ for every $t\in [a,b]$ by
$$
e_p(t,a) = e^{\int_{[a,t)}\tilde{p}(s)\,d\mu_g(s)},
$$
where
\begin{equation}\label{eq:exp:h tilde}
\tilde{p}(t) = \begin{dcases}
p(t) &\text{if $t \in [a,b]\setminus D_g$,}
\\
\frac{\log\big(1 + p(t)(g(t^+) - g(t))\big)}{g(t^+)-g(t)} &\text{if $t \in [a,b) \cap D_g$.}
\end{dcases}
\end{equation}
\end{defn}
In particular, given $p \in \mathcal{L}^1_g([a,b),\mathbb{R})$ a function satisfying Condition~\eqref{eq:exp:nonreasonance Condition}, then $\tilde{p}\in\mathcal{L}^1_g([a,b),\mathbb{R})$, and $e_{p}(\cdot, a)\in \mathcal{AC}_g([a,b),\mathbb{R})$, the reader is referred to~\cite[Lemmata~6.2 and~6.3]{FP} and improvements in~\cite[Theorem~3.2]{Mar}.

Now, we appeal to the generalization of the Gr\"{o}nwall Lemma to the Stieltjes derivative introduced by Larivière in~\cite[Proposition~4.1.4]{ThesisFL}, and further generalized in~\cite[Theorem~5.4]{GallMarSlav2025}. This lemma will play a crucial role in establishing global solutions defined on the positive
real half-line as we shall prove in the following section.
\begin{lem}\label{lem:g-Gronwall}
Let $u\in \mathcal{AC}_{g}([a,b],\mathbb{R})$. Assume that there exist functions $k,p\in \mathcal{L}^{1}_{g}([a,b),\mathbb{R})$, satisfying $1+p(t)\mu_g(\{t\})>0$ for all $t\in [a,b)\cap D_g$,  such that
  $$
  u_{g}'(t)\le k(t)+ p(t) u(t)\quad \text{for $g$-almost all }t\in[a,b).
  $$
  Then,
  $$
 u(t)\le e_p(t,a) \Big(\int_{[a,t)} \frac{e_{p}^{-1}(s,a) k(s)}{1+p(s)\mu_g(\{s\})}\, d\mu_{g}(s)+u(a)\Big), \quad t\in [a,b].
  $$
\end{lem}


\section{Prolongation of solutions and maximal interval of existence}
Let $O$ be a nonempty open set of $\mathbb{R}^n$ with respect to the usual topology $\tau_{\mathbb{R}^n}$, and $I$ a $g$-open set of $\mathbb{R}$ containing $t_0\ge 0$ with $\sup I > t_0$. If we set $ \Omega :=I\times O$, then $\Omega$ is an open set with respect to the topology $\hat{\tau}_{g}$ on $\mathbb{R}\times\mathbb{R}^n$, generated by the open sets $U\times V$ such that $U\in \tau_g$ and $V\in \tau_{\mathbb{R}^n}$. Here and afterwards, $B_{\mathbb{R}^n}(\mathbf{x},\delta)$ denotes the open ball of $\mathbb{R}^n$ centered at $\mathbf{x} \in\mathbb{R}^n$ with radius $\delta>0$.

 Let us consider the Stieltjes dynamical system:
\begin{equation}\label{eq:dyn-system-Prolongation}
\begin{aligned}
     \mathbf{x}_{g}'(t) &= \mathbf{f}(t,\mathbf{x}(t))  \quad \text{for $g$-almost all $t\ge t_0$, $t\in I$},\\
     \mathbf{x}(t_0)&=\mathbf{x}_0 \in O,
\end{aligned}
\end{equation}
where $\mathbf{f}=(f_1,\dots,f_n):\Omega\cap\Big([t_0,\infty)\times \mathbb{R}^n\Big) \to \mathbb{R}^n$ and $\Omega$ defined above satisfying the following assumptions:
\begin{enumerate}
  \item[({\rm H$_{\Omega}$})] For every $(t,\mathbf{x}) \in \Omega$ with $t\ge t_0$,
  \begin{enumerate}
          \item [(a)]  one of the following conditions hold:
            \begin{enumerate}
          \item [(a1)] there exists $\delta>0$ such that $(t-\delta,t+\delta)\times B_{\mathbb{R}^n}(\mathbf{x},\delta) \subset \Omega$;
          \item [(a2)] if for every $\delta>0$, $(t-\delta,t+\delta)\times B_{\mathbb{R}^n}(\mathbf{x},\delta) \not\subset \Omega$, then $t\in D_g$ and there exists $\epsilon>0$ such that $(t-\epsilon,t]\times B_{\mathbb{R}^n}(\mathbf{x},\delta)\subset \Omega$;\end{enumerate}
          \item [(b)] $(t_0,\mathbf{x}_{\mathbf{f},t_0}^{+})\in \Omega$, and if $(t,\mathbf{x}) \in\Omega\cap (D_g\times \mathbb{R}^n)$ such that $(t,\mathbf{x}_{\mathbf{f},t}^{+})\in \Omega$, then $(t,\mathbf{x}_{\mathbf{f},t}^{+})$ satisfies Condition~{\rm(H$_{\Omega}$)(a)(1)}. Here and afterward, the notation $\mathbf{x}_{\mathbf{f},t}^{+}$ refers to
          $$
          \mathbf{x}_{\mathbf{f},t}^{+}:=\mathbf{x}+\mu_g(\{t\})\mathbf{f}(t,\mathbf{x});
          $$
        \end{enumerate}
  \item[({\rm H$_{\mathbf{f}}$})]
\begin{enumerate}[(i)]
\item   for all $\mathbf{x}\in O$, $\mathbf{f}(\cdot,\mathbf{x})$ is $g$-measurable;
\item  $\mathbf{f}(\cdot,\mathbf{x}_0)\in \mathcal{L}^{1}_{g,loc}([t_0,\infty),\mathbb{R}^n)$;
\item  $\mathbf{f}$ is {\it $g$-integrally locally Lipschitz continuous}, i.e.  for every $r>0$, there exists a function $L_r\in \mathcal{L}^{1}_{g,loc}([t_0,\infty),[0,\infty))$ such that
$$
\|\mathbf{f}(t,\mathbf{x})-\mathbf{f}(t,\mathbf{y})\|\le L_r(t) \|\mathbf{x}-\mathbf{y}\|,
$$
for $g$-almost all $t\in I \cap [t_0,\infty)$ and all $\mathbf{x},\mathbf{y}\in \overline{B_{\mathbb{R}^n}(\mathbf{x}_0,r)}\cap O$.
\end{enumerate}
\end{enumerate}
We recall the local existence result~\cite[Theorem~7.4]{FP}.
\begin{thm}\label{thm:loc-existence}
Assume that the conditions in~{\rm(H$_{\mathbf{f}}$)} hold. Then there exists $\tau>0$ such that the system~\eqref{eq:dyn-system-Prolongation} has a unique solution $\mathbf{x}\in \mathcal{AC}_{g}([t_0,t_0+\tau],\mathbb{R}^n)$.
\end{thm}
In the sequel, as shown in~\cite[Section~4.2]{ThesisFL}, the solution given by Theorem~\ref{thm:loc-existence} can be extended to intervals larger than $[t_0,t_0+\tau]$ up to a maximal interval as long as the graph of the solution does not leave $\Omega$ and the right-hand side limit of the solution at $t_0+\tau$ belongs to $O$. It should be noted that the derivator $g$ may not necessarily be continuous at $t_0$. For this purpose, Condition~{\rm(H$_{\Omega}$)} is necessary to seek the maximal interval of existence of the solution. Indeed, notice that Condition~{\rm(H$_{\Omega}$)} guarantees existence of a local solution on an interval $[t_0,t_0+\tau]$ even in the case where $t_0\in D_g$, and insures extension of the solution $\mathbf{x}$ on some larger interval $[t_0,t_0+\epsilon]$, $\epsilon>\tau$ if $t_0+\tau \in D_g$  and $\mathbf{x}(t_0+\tau^+) \in O$.

Let us define the set
$$
\mathcal{S}(t_0,\mathbf{x}_0):=\{\mathbf{x}:I_\mathbf{x} = J_\mathbf{x}\cap[t_0,\infty)\to\mathbb{R}^n: \text{$\mathbf{x}$ is a solution of the system~\eqref{eq:dyn-system-Prolongation}}\},
$$
where $J_\mathbf{x}$ is a $g$-open interval containing $t_0$ with $\sup J_\mathbf{x} > t_0$. In the sequel, we adopt the notation $t_{sup}:=\sup I_\mathbf{x}$.

\begin{defn}\label{df:x extendible to the right}
Let $\mathbf{x},\mathbf{y}\in \mathcal{S}(t_0,\mathbf{x}_0)$.
\begin{enumerate}[(1)]
  \item We say that $\mathbf{x}$ is {\it smaller} than $\mathbf{y}$ (and we denote $\mathbf{x}\prec \mathbf{y}$), if and only if
\begin{enumerate}[(i)]
  \item $I_\mathbf{x} \subset I_\mathbf{y}$;
  \item $\sup I_\mathbf{y}> \sup I_\mathbf{x}$;
  \item $\mathbf{y}|_{I_\mathbf{x}}=\mathbf{x}$.
\end{enumerate} In this case, we say that $\mathbf{x}$ is {\it extendible to the right} and $\mathbf{y}$ is {\it a prolongation to the right} of~$\mathbf{x}$.
   \item We write that $\mathbf{x}\preceq \mathbf{y} \iff \mathbf{x}\prec \mathbf{y} \text{ or }\mathbf{x}=\mathbf{y}$.
   \end{enumerate}
\end{defn}
\begin{rem}\label{rem:adding the point x(t_sup^-) does not properly prolongate x}
 It is worth mentioning that given a solution $\mathbf{x}:[t_0,t_1)\to \mathbb{R}^n$ such that $\mathbf{x}(t_1^-):=\lim\limits_{t\to t_1^-}\mathbf{x}(t) \in O$ and $t_1\in D_g$, to not increase the notation, we will replace $\mathbf{x}$ by the function
 $\mathbf{x}:[t_0,t_1]\to \mathbb{R}^n$ defined by
$$
\mathbf{x}=\begin{dcases}
                          \mathbf{x}(t)& \mbox{if } t\in [t_0,t_1), \\
                         \mathbf{x}(t_1^-), & \mbox{if } t=t_1.
                       \end{dcases}
$$
\end{rem}

The next result asserts that two prolongations to the right of a solution are equal on the common interval of existence, the reader is referred to~\cite[Theorem~4.2.3]{ThesisFL} for the proof.
\begin{thm}\label{thm:prolongations coincide on common interval}
Assume that~{\rm(H$_{\mathbf{f}}$)} holds. Let $\mathbf{x},\mathbf{y},\mathbf{z}\in \mathcal{S}(t_0,\mathbf{x}_0)$ be such that $\mathbf{y}$ and~$\mathbf{z}$ are two prolongations of $\mathbf{x}$ then $\mathbf{y}=\mathbf{z}$ on $I_\mathbf{y}\cap I_\mathbf{z}$.
\end{thm}

In~\cite[Theorem~4.2.4]{ThesisFL}, extendible solutions to the right were characterized as follows.

\begin{thm}\label{thm:extendible solution to the right}
 Assume that~{\rm(H$_{\mathbf{f}}$)} holds.   Let $\mathbf{x}\in \mathcal{S}(t_0,\mathbf{x}_0)$. The following assumptions are equivalent:
  \begin{enumerate}[$(1)$]
    \item $\mathbf{x}$ is extendible to the right;
    \item
    \begin{enumerate}[{\rm(i)}]
           \item $\text{Graph}(\mathbf{x}):=\{(t,\mathbf{x}(t)): t\in I_\mathbf{x}\}$ is bounded;
           \item $A\cup A^+ \subset \Omega$ where
           $$
           A=\{(t_{sup},\mathbf{s})\in [t_0,\infty)\times \mathbb{R}^n: \exists\{t_n\}_n\subset I_\mathbf{x},\, t_n \nearrow t_{sup} \text{ and } \mathbf{s}=\lim\limits_{n\to\infty}\mathbf{x}(t_n)\},
           $$
           $$
           A^+=\{(t_{sup},\mathbf{s}_{\mathbf{f},t_{sup}}^{+}):(t_{sup},\mathbf{s})\in A\}.
           $$
         \end{enumerate}
  \end{enumerate}
\end{thm}

\begin{defn}
Let $\mathbf{x}\in \mathcal{S}(t_0,\mathbf{x}_0)$. We say that $\mathbf{x}$ is a {\it  maximal  solution} of~\eqref{eq:dyn-system-Prolongation} defined on an interval $I_{\mathbf{x}}$ if,  for every $\mathbf{y}\in \mathcal{S}(t_0,\mathbf{x}_0)$ satisfying $\mathbf{x}\preceq \mathbf{y}$, we have $\mathbf{x}=\mathbf{y}$. $I_{\mathbf{x}}$ is referred to as the {\it maximal interval of existence}.
\end{defn}

As shown in~\cite[Theorem~4.2.5]{ThesisFL}, the existence of the maximal solution holds as a consequence of using Zorn's lemma~\cite{Zorn}. For this purpose, let us consider $\mathbf{x}\in \mathcal{S}(t_0,\mathbf{x}_0)$, and define the set $\mathcal{X}$:
$$
\mathcal{X}=\{\mathbf{y}\in  \mathcal{S}(t_0,\mathbf{x}_0): \mathbf{x}\preceq\mathbf{y}\}.
$$
$\mathcal{X}$ is a nonempty partially ordered set since $\mathbf{x}\in\mathcal{X}$. Now, following the same argument as in the proof of~\cite[Theorem~4.2.5]{ThesisFL}, by using the partial order $\preceq$ defined in Definition~\ref{df:x extendible to the right} instead, we deduce that every chain of $\mathcal{X}$ has the largest element. This yields the existence of a maximal solution defined on $\bigcup_{\mathbf{y}\in \mathcal{X}}I_{\mathbf{y}}$, constructed by taking the union of all the prolongations to the right of $\mathbf{x}$. In addition, Theorem~\ref{thm:prolongations coincide on common interval} guarantees the uniqueness of the maximal solution, and we obtain the following theorem.

\begin{thm}\label{thm:existence of maximal solution}
Assume that~{\rm(H$_{\Omega}$)} and~{\rm(H$_{\mathbf{f}}$)} hold. There exists a unique maximal solution $\mathbf{x} \in\mathcal{S}(t_0,\mathbf{x}_0)$ such that $\omega(t_0,\mathbf{x}_0):=\sup I_{\mathbf{x}}\le \infty$.
\end{thm}

The next theorem highlights three alternative cases that occur, the reader is referred to~\cite[Theorem~4.2.6]{ThesisFL} for the proof.
\begin{thm}\label{thm:3-alternatives of max-interval-existence}
Assume that~{\rm(H$_{\Omega}$)} and~{\rm(H$_{\mathbf{f}}$)} hold. Let  $\mathbf{x}\in\mathcal{S}(t_0,\mathbf{x}_0)$ be the maximal solution of~\eqref{eq:dyn-system-Prolongation}, then one of the alternatives holds:
  \begin{enumerate}
 \item[{{\rm(A1)}}] $\omega(t_0,\mathbf{x}_0)=\infty$;
 \item[{{\rm(A2)}}] $\omega(t_0,\mathbf{x}_0)<\infty$, and for every $\{t_n\}_n \subset I_{\mathbf{x}}$ such that $t_n \nearrow \omega(t_0,\mathbf{x}_0)$,  $\{\mathbf{x}(t_n)\}_n$ is not bounded;
 \item[{{\rm(A3)}}] $\omega(t_0,\mathbf{x}_0)<\infty$, moreover, there exists $\{t_n\}_n \subset I_{\mathbf{x}}$ satisfying $t_n \nearrow \omega(t_0,\mathbf{x}_0)$ and $\{\mathbf{x}(t_n)\}_n$ is a bounded sequence, such that for every  subsequence $\{t_{n_k}\}_k$ verifying $\mathbf{x}(t_{n_k}) \to \mathbf{m}$, we have that
     $$
      \{(\omega(t_0,\mathbf{x}_0), \mathbf{m}),(\omega(t_0,\mathbf{x}_0),\mathbf{m}_{\mathbf{f},\omega(t_0,\mathbf{x}_0)}^{+})\} \not\subset \Omega.
     $$
  \end{enumerate}
\end{thm}
\begin{rem}\label{rem:3 alternatives for max interv}
  It is worth mentioning that the maximal solution $\mathbf{x}$ of~\eqref{eq:dyn-system-Prolongation} shall be $g$-absolutely continuous on every interval $[a,b]\subset I_{\mathbf{x}}$, which immediately  holds from the three alternatives of Theorem~\ref{thm:3-alternatives of max-interval-existence}. Notice that if Alternative~{\rm(A1)} holds then $I_{\mathbf{x}}=[t_0,\infty)$, if Alternative~{\rm(A2)} or~{\rm(A3)} holds, then  $I_{\mathbf{x}}=[t_0,\omega(t_0,\mathbf{x}_0))$ or $I_{\mathbf{x}}=[t_0,\omega(t_0,\mathbf{x}_0)]$.
\end{rem}

 Thanks to Theorem~\ref{thm:3-alternatives of max-interval-existence} established by Larivière, we obtain the following corollary, which ensures the global existence of the maximal solution over the whole interval $[t_0,\infty)$ in the case where $[t_0,\infty)\subset I$.

\begin{cor}\label{cor:compacity yields tm=infty}
Assume that~{\rm(H$_{\Omega}$)} and~{\rm(H$_{\mathbf{f}}$)} hold. Let  $\mathbf{x}\in\mathcal{S}(t_0,\mathbf{x}_0)$ be the maximal solution of~\eqref{eq:dyn-system-Prolongation}. If
$$
\Omega^+:=\{(t,\mathbf{u}) \in \Omega: \mathbf{u}_{\mathbf{f},t}^{+}\in \Omega \}=\Omega,
$$
and there exists a compact set $D\subset O$ such that $\mathbf{x}(t) \in D$ for every $t\in I_{\mathbf{x}}$, then $\omega(t_0,\mathbf{x}_0)=\infty$.
\end{cor}
\begin{proof}
Let us assume that $\omega(t_0,\mathbf{x}_0)<\infty$. Since $D$ is compact, for all $\{t_n\}_n \subset I_{\mathbf{x}}$ such that $t_n \nearrow \omega(t_0,\mathbf{x}_0)$, $\{\mathbf{x}(t_n)\}_n$ is bounded.  Thus, it results from Theorem~\ref{thm:3-alternatives of max-interval-existence} that Alternative~{\rm(A3)} holds. Therefore, there exists $\{t_n\}_n \subset I_{\mathbf{x}}$ with $t_n \nearrow \omega(t_0,\mathbf{x}_0)$ and $\{(t_n ,\mathbf{x}(t_n))\}_n$ is bounded, such that for a fixed subsequence $\{t_{n_k}\}_k$ verifying $\mathbf{x}(t_{n_k}) \to \mathbf{m}$, we have that
  $$
  \{(\omega(t_0,\mathbf{x}_0), \mathbf{m}),(\omega(t_0,\mathbf{x}_0),\mathbf{m}_{\mathbf{f},\omega(t_0,\mathbf{x}_0)}^{+})\} \not\subset\Omega.
   $$
   On the other hand,  $D$ being a compact set yields that
  $$
  (t_{n_k},\mathbf{x}(t_{n_k})) \to (\omega(t_0,\mathbf{x}_0),\mathbf{m}) \in [t_0,\omega(t_0,\mathbf{x}_0)] \times D \subset \Omega.
  $$
 Now, since $\Omega^+=\Omega$, we deduce that $(\omega(t_0,\mathbf{x}_0),\mathbf{m}_{\mathbf{f},\omega(t_0,\mathbf{x}_0)}^{+}) \in \Omega$, which contradicts
 $$
 \{(\omega(t_0,\mathbf{x}_0), \mathbf{m}),(\omega(t_0,\mathbf{x}_0),\mathbf{m}_{\mathbf{f},\omega(t_0,\mathbf{x}_0)}^{+})\} \not\subset\Omega.
 $$
  Hence, $\omega(t_0,\mathbf{x}_0)=\infty$.
\end{proof}

Now based on Theorems~\ref{thm:extendible solution to the right} and~\ref{thm:3-alternatives of max-interval-existence}, we provide a characterization of the maximal solution of the problem~\eqref{eq:dyn-system-Prolongation}.
\begin{thm}\label{thm:characterize max sol}
 Assume that~{\rm(H$_{\Omega}$)} and~{\rm(H$_{\mathbf{f}}$)} hold. Let $\mathbf{x}\in\mathcal{S}(t_0,\mathbf{x}_0)$. The following assumptions are equivalent:
  \begin{enumerate}[$(1)$]
    \item $\mathbf{x}$ is maximal;
    \item for every compact set $K\subset \Omega$, there exists $t_K \in I_\mathbf{x}$ such that
    $$
    \{(t,\mathbf{x}(t)),(t,\mathbf{x}(t)+\mu_g(\{t\})\mathbf{f}(t,\mathbf{x}(t)))\}\not\subset K,
    $$
 for all $t\ge t_K$, $t\in I_\mathbf{x}$.
  \end{enumerate}
\end{thm}

\begin{proof}
 Assume that  $\mathbf{x}$ is maximal. By Theorem~\ref{thm:3-alternatives of max-interval-existence}, two cases may occur:

\textbf{Case 1:} if $\omega(t_0,\mathbf{x}_0)=\sup I_{\mathbf{x}}=\infty$,  by contradiction assume that there exist a compact set $K \subset \Omega$ and a sequence $\{t_n\}_n \subset I_\mathbf{x}$ such that
$$
t_n\nearrow \omega(t_0,\mathbf{x}_0), \,  \text{ and }     \{(t_n,\mathbf{x}(t_n)),(t_n,\mathbf{x}(t_n)+\mu_g(\{t_n\})\mathbf{f}(t_n,\mathbf{x}(t_n)))\} \subset K \text{ for all } n\in\mathbb{N}.
$$
This implies that $\{(t_n,\mathbf{x}(t_n))\}_n$ is bounded. Therefore, there exists a convergent subsequence $\{(t_{n_k},\mathbf{x}(t_{n_k}))\}_k$  such that
$$
(t_{n_k},\mathbf{x}(t_{n_k})) \to (\tau,\mathbf{u})\in K \subset [t_0,\infty)\times O.
$$
Particularly, $t_{n_k} \nearrow \tau$ which contradicts $t_n \nearrow \omega(t_0,\mathbf{x}_0)=\infty$.

\textbf{Case 2:} if $\omega(t_0,\mathbf{x}_0)< \infty$, then by Theorem~\ref{thm:3-alternatives of max-interval-existence}, there are two subcases:

\textbf{Subcase 1:} if Alternative~{\rm(A2)} holds, then $I_\mathbf{x}=[t_0,\omega(t_0,\mathbf{x}_0))$ and $\|\mathbf{x}(t)\| \to \infty$ as $t \nearrow \omega(t_0,\mathbf{x}_0)$. Thus, for every $M>0$ there exists $t^* \in I_\mathbf{x}$ such that $\|\mathbf{x}(t)\| \ge M$ for all $t\ge t^*$ with $t\in I_\mathbf{x}$. Hence, for every compact $K \subset \Omega$ there exists $t_K \in I_\mathbf{x}$ such that for all $t\ge t_K$ with $t\in I_\mathbf{x}$, we have $(t,\mathbf{x}(t))\notin K$, in particular $\{(t,\mathbf{x}(t)), (t,\mathbf{x}(t)+\mu_g(\{t\})\mathbf{f}(t,\mathbf{x}(t)))\} \not\subset K$.

\textbf{Subcase 2:} if Alternative~{\rm(A3)} holds, then we distinguish two cases. If $I_\mathbf{x}=[t_0,\omega(t_0,\mathbf{x}_0))$, then  $\text{Graph}(\mathbf{x})$ approaches the boundary of $\Omega$ and $$
(\omega(t_0,\mathbf{x}_0),\mathbf{x}(\omega(t_0,\mathbf{x}_0)^-))\notin \Omega.
$$
Thus, for every compact $K \subset \Omega$, there exists $t_K \in [t_0,\omega(t_0,\mathbf{x}_0))$ such that, for all $t\ge t_K$ with $t\in I_\mathbf{x}$, $(t,\mathbf{x}(t)) \notin K$, which yields
$$
\{(t,\mathbf{x}(t)),(t,\mathbf{x}(t)+\mu_g(\{t\})\mathbf{f}(t,\mathbf{x}(t)))\} \not\subset K.
$$
Now, if  $I_\mathbf{x}=[t_0,\omega(t_0,\mathbf{x}_0)]$, then $\omega(t_0,\mathbf{x}_0) \in D_g$, $(\omega(t_0,\mathbf{x}_0),\mathbf{x}(\omega(t_0,\mathbf{x}_0)))\in \Omega$ and
 $$
(\omega(t_0,\mathbf{x}_0),\mathbf{x}(\omega(t_0,\mathbf{x}_0))+\mu_g(\{\omega(t_0,\mathbf{x}_0)\})\mathbf{f}(\omega(t_0,\mathbf{x}_0), \mathbf{x}(\omega(t_0,\mathbf{x}_0))))\notin\Omega.
$$
 Thus, for every compact $K \subset \Omega$, there exists $t_K \in [t_0,\omega(t_0,\mathbf{x}_0)]$ such that, for all $t\ge t_K$ with $t\in I_\mathbf{x}$, $(t,\mathbf{x}(t)+\mu_g(\{t\})\mathbf{f}(t,\mathbf{x}(t))) \notin K$, which yields
$$
\{(t,\mathbf{x}(t)),(t,\mathbf{x}(t)+\mu_g(\{t\})\mathbf{f}(t,\mathbf{x}(t)))\} \not\subset K.
$$

Conversely, by contradiction, let us assume that $\mathbf{x}:I_{\mathbf{x}}\to \mathbb{R}^n$ is not maximal. Thus, $t_{sup}=\sup I_{\mathbf{x}} < \infty$ and  $\mathbf{x}$ is extendible to the right. From Theorem~\ref{thm:extendible solution to the right}, it follows that $\text{Graph}(\mathbf{x})$ is bounded and $A\cup A^+ \subset \Omega$. Thus, $[t_0,t_{sup}]\times\overline{\text{Graph}(\mathbf{x})}$ is compact, $\mathbf{x}(t_{sup}^-)=\mathbf{x}(t_{sup})=\mathbf{s} \in O$ and $\mathbf{s}_{\mathbf{f},t_{sup}}^{+} \in O$. Consequently,  by Theorem~\ref{thm:loc-existence}, there exists a prolongation to the right $\hat{\mathbf{x}}:[t_0,t_{sup} +\epsilon) \to \mathbb{R}^n$ such that $\hat{\mathbf{x}}|_{I_\mathbf{x}}=\mathbf{x}$. Therefore, for the compact $K=\overline{\text{Graph}(\mathbf{x})}\cup \{(t_{sup},\mathbf{s}_{\mathbf{f},t_{sup}}^{+})\}\subset\Omega$, we have that
$$
\{(t,\mathbf{x}(t)),(t,\mathbf{x}(t)+\mu_g(\{t\})\mathbf{f}(t,\mathbf{x}(t)))\} \subset K \text{ for all } t\in I_\mathbf{x},
$$
which yields a contradiction. Hence, $\mathbf{x}$ is maximal.
\end{proof}

The logical negation of Theorem~\ref{thm:characterize max sol} also provides an interesting characterization of extendible solutions.

\begin{cor}\label{cor:characterize extendible sol}
Assume that~{\rm (H$_{\Omega}$)} and~{\rm(H$_{\mathbf{f}}$)} hold. Let $\mathbf{x}\in\mathcal{S}(t_0,\mathbf{x}_0)$. The following assumptions are equivalent:
\begin{enumerate}[$(1)$]
\item $\mathbf{x}$ extendible to the right;
\item there exist a compact set $K\subset \Omega$, and a sequence $\{t_n\}_n\subset I_\mathbf{x}$ with $t_n \nearrow \omega(t_0,\mathbf{x}_0)$ such that
    $$
    \{(t_n,\mathbf{x}(t_n)),(t_n,\mathbf{x}(t_n)+\mu_g(\{t_n\})\mathbf{f}(t_n,\mathbf{x}(t_n)))\}\subset K, \text{ for all }n\in \mathbb{N}.
    $$
  \end{enumerate}
\end{cor}

Given the generalization of the Gr\"{o}nwall lemma, Lemma~\ref{lem:g-Gronwall}, we can state the next theorem which provides the global existence of the solution over $[t_0,\infty)$ and a priori bound of the solution when $[t_0,\infty) \subset I$ and $O=\mathbb{R}^n$.

\begin{thm}\label{thm:global-existence-with-linear-growth-condition}
Assume that~{\rm (H$_{\Omega}$)} and~{\rm(H$_{\mathbf{f}}$)} hold. Let  $[t_0,\infty) \subset I$ and $O=\mathbb{R}^n$ and assume that the conditions in~{\rm(H$_{\mathbf{f}}$)} hold for $\mathbf{f}:[t_0,\infty)\times \mathbb{R}^n\to \mathbb{R}^n$ satisfying the linear growth condition:
\begin{enumerate}
  \item [{{\rm(H$_{LG}$)}}] there exist functions
  $k\in \mathcal{L}^{1}_{g,loc}([t_0,\infty),\mathbb{R})$, and $p\in \mathcal{L}^{1}_{g,loc}([t_0,\infty),[0,\infty))$ such that
$$
\|\mathbf{f}(t,\mathbf{u})\| \le k(t)+ p(t)\|\mathbf{u}\|,
$$
for $g$-almost all $t\in[t_0,\infty)$ and all $\mathbf{u}\in O$.
  \end{enumerate}
 If $\mathbf{x} \in \mathcal{S}(t_0,\mathbf{x}_0)$ is the maximal solution of~\eqref{eq:dyn-system-Prolongation},  then
\begin{equation}\label{eq:priori bound of global sol}
  \|\mathbf{x}(t)\|\le e_p(t,t_0) \Big(\int_{[t_0,t)} \frac{e_{p}^{-1}(s,t_0) k(s)}{1+p(s)\mu_g(\{s\})}\, d\mu_g(s)+\|\mathbf{x}_0\|\Big),
\end{equation}
for all $t\in [t_0,T]$, $T\in I_{\mathbf{x}}$. Moreover, we have $I_{\mathbf{x}}=[t_0,\infty)$.
\end{thm}
\begin{proof}
Let $\mathbf{x} \in\mathcal{S}(t_0,\mathbf{x}_0)$ be the maximal solution of~\eqref{eq:dyn-system-Prolongation} defined on the maximal interval $I_{\mathbf{x}}$ with $\omega(t_0,\mathbf{x}_0)=\sup I_{\mathbf{x}}$. By~{\rm(H$_{LG}$)}, we have for $T>t_0$ with $T\in I_{\mathbf{x}}$ that
$$
(\|\mathbf{x}\|)_g'(t) \le \|\mathbf{x}_g'(t)\|=\|\mathbf{f}(t,\mathbf{x}(t))\| \le k(t)+ p(t)\|\mathbf{x}(t)\| \quad \text{for $g$-almost all } t\in [t_0,T).
$$
Observe that $\|\cdot\|$ is Lipschitz continuous on $\mathbb{R}^n$. Thus, by Lemma~\ref{lem:Lipschitz o g-absolute func}, $\|\mathbf{x}(\cdot)\| \in \mathcal{AC}_g([t_0,T],\mathbb{R})$. Using the generalized version of the Gr\"{o}nwall Lemma, Lemma~\ref{lem:g-Gronwall}, we obtain
$$
\|\mathbf{x}(t)\|\le  e_p(t,t_0) \Big(\int_{[t_0,t)} \frac{e_{p}^{-1}(s,t_0) k(s)}{1+p(s)\mu_g(\{s\})}\, d\mu_g(s)+\|\mathbf{x}_0\|\Big) \quad\text{ for all } t\in [t_0,T].
$$
Assume by contradiction that $\omega(t_0,\mathbf{x}_0)<\infty$. Thus, for all  $t\in [t_0,T]$, we obtain a larger bound of $\mathbf{x}$:
\begin{equation*}
     \|\mathbf{x}(t)\|  \le \sup_{t\in [t_0,T]} e_p(t,t_0) \Big(\int_{[t_0,t)} \frac{e_{p}^{-1}(s,t_0) |k(s)|}{1+p(s)\mu_g(\{s\})}\, d\mu_g(s)+\|\mathbf{x}_0\|\Big).
\end{equation*}
Since  $p\in \mathcal{L}^{1}_{g,loc}([t_0,\infty),[0,\infty))$, then  $e_{p}(t,t_0)(1+p(t)\mu_g(\{t\}))\ge 1$ for all $t\in [t_0,T]$. Consequently,
\begin{align*}
     \|\mathbf{x}(t)\|&\le \sup_{t\in [t_0,T]}e_p(t,t_0) \Big(\int_{[t_0,t)} |k(s)|\, d\mu_g(s)+\|\mathbf{x}_0\|\Big)  \\
   &=  e_p(T,t_0) \Big(\|k\|_{\mathcal{L}^{1}_{g}([t_0,T),\mathbb{R})}+\|\mathbf{x}_0\|\Big).
\end{align*}
As $T\nearrow \omega(t_0,\mathbf{x}_0) <\infty$, $\mathbf{x}(t) \in B_{\mathbb{R}^n}(\mathbf{0},M_1)$ for all $t\in I_{\mathbf{x}}$ where
$$
M_1=e_p(\omega(t_0,\mathbf{x}_0),t_0) \Big(\|k\|_{\mathcal{L}^{1}_{g}([t_0,\omega(t_0,\mathbf{x}_0)),\mathbb{R})}+\|\mathbf{x}_0\|\Big).
$$
As $\mathbf{x}(\omega(t_0,\mathbf{x}_0)^-)=\mathbf{m} \in \mathbb{R}^n$, then $\mathbf{m}_{\mathbf{f},\omega(t_0,\mathbf{x}_0)}^{+} \in \mathbb{R}^n$. Therefore, for the compact set $K=[t_0, \omega(t_0,\mathbf{x}_0)]\times B_{\mathbb{R}^n}(\mathbf{0},M)$ with $M=\max\{M_1, \|\mathbf{m}_{\mathbf{f},\omega(t_0,\mathbf{x}_0)}^{+}\|\}$, we have that
$$
\{(t,\mathbf{x}(t)),(t,\mathbf{x}(t)+\mu_g(\{t\})\mathbf{f}(t,\mathbf{x}(t)))\} \subset K \text{ for all } t\in I_{\mathbf{x}}.
$$
By Corollary~\ref{cor:characterize extendible sol}, we obtain that $\mathbf{x}$ is extendible to the right which is a contradiction. Hence, $\omega(t_0,\mathbf{x}_0)=\infty$.
\end{proof}

\section{Lyapunov-like stability results}

In this section, we introduce Lyapunov-type results in the context of Stieltjes dynamical systems, based on the classical Lyapunov's second method, considering Stieltjes dynamical systems of the form:
\begin{equation}\label{eq:stability-dyn-system}
\begin{aligned}
     \mathbf{x}_{g}'(t) &= \mathbf{f}(t,\mathbf{x}(t))  \quad \text{for $g$-almost all $t\ge t_0\ge 0$, $t\in I$},\\
\end{aligned}
\end{equation}
where $\mathbf{f}=(f_1,\dots,f_n):\mathbb{R}^+\times B_{\mathbb{R}^n}(\mathbf{0},r_0)\to\mathbb{R}^n$ satisfying $\mathbf{f}(t,\mathbf{0})=\mathbf{0}$ for all $t\ge 0$. This study permits to draw conclusions about the local behavior of solutions of the dynamical system~\eqref{eq:stability-dyn-system} around the equilibrium $\mathbf{x}=\mathbf{0}$, which is called in the stability literature the {\it trivial solution}. Before stating the assumptions on $\mathbf{f}$ to guarantee the existence and the uniqueness of solution of the system~\eqref{eq:stability-dyn-system} starting at some $t_0\ge 0$, notice that if $t_0 \in D_g$, and $\mathbf{x}_0 \in B_{\mathbb{R}^n}(\mathbf{0},r_0)$ such that $\mathbf{x}_0+\mu_g(\{t_0\})\mathbf{f}(t_0,\mathbf{x}_0) \notin B_{\mathbb{R}^n}(\mathbf{0},r_0)$, then existence of a solution $\mathbf{x}$ satisfying $\mathbf{x}(t_0)=\mathbf{x}_0$ cannot be guaranteed. Thus, we assume
\begin{enumerate}
  \item [{{\rm(H$_{r}$)}}]there exists $r\in(0,r_0]$ such that
  for all $(t,\mathbf{u}) \in (\mathbb{R}^+ \cap D_g)\times B_{\mathbb{R}^n}(\mathbf{0},r)$, we have
  $$
   \mathbf{u}+\mu_g(\{t\}) \mathbf{f}(t,\mathbf{u}) \in B_{\mathbb{R}^n}(\mathbf{0},r_0),
  $$
in other terms, for all $t\in (\mathbb{R}^+ \cap D_g)$,
  $$
   B_{\mathbb{R}^n}(\mathbf{0},r)\subset\{\mathbf{u}\in B_{\mathbb{R}^n}(\mathbf{0},r_0): \mathbf{u}+\mu_g(\{t\}) \mathbf{f}(t,\mathbf{u}) \in B_{\mathbb{R}^n}(\mathbf{0},r_0) \}.
  $$
\end{enumerate}
Now, let us assume that $\mathbf{f}$ fulfills~{\rm(H$_\mathbf{f}$)}, where $\Omega$ is an open set of the form $\Omega=I\times B_{\mathbb{R}^n}(\mathbf{0},r)$ satisfying~{\rm(H$_\Omega$)} as in  Section~5.1 with $I$ a $g$-open set containing the whole $\mathbb{R}^+$.

Therefore, under hypotheses~{\rm(H$_\Omega$)},~{\rm(H$_\mathbf{f}$)}, and~{\rm(H$_{r}$)}, it follows from Theorem~\ref{thm:existence of maximal solution} that, for every $(t_0,\mathbf{x}_0)\in \mathbb{R}^+\times B_{\mathbb{R}^n}(\mathbf{0},r)$, there exists a unique maximal solution $\mathbf{x} = \mathbf{x}(\cdot,t_0,\mathbf{x}_0)\in\mathcal{S}(t_0,\mathbf{x}_0) \cap \mathcal{AC}_{g,loc}(I_{t_0,\mathbf{x}_0}, \mathbb{R}^n)$ of~\eqref{eq:stability-dyn-system} defined on a maximal interval of existence that we denote $I_{t_0,\mathbf{x}_0}$ since it depends on the initial data $(t_0,\mathbf{x}_0)$. As before, we denote $\omega(t_0,\mathbf{x}_0)=\sup I_{t_0,\mathbf{x}_0} \le \infty$.
\subsection{Lyapunov stability notions}
In this subsection, we present stability concepts within the framework of Stieltjes' differentiation. Through illustrative examples, we highlight the influence of the sets $C_g$ and $D_g$ on the change of stability properties.
\begin{defn}
  The trivial solution $\mathbf{x}=\mathbf{0}$ of the system~\eqref{eq:stability-dyn-system} is said to be
  \begin{itemize}
    \item[$\bullet$] {\it stable} if, for all $\epsilon >0$ and $t_0\in \mathbb{R}^+$, there exists $\delta=\delta(\epsilon,t_0)\in(0,r)$ such that
    $$
    \|\mathbf{x}_0\|<\delta \quad \text{ implies that }\quad \|\mathbf{x}(t,t_0,\mathbf{x}_0)\|< \epsilon \text{ for all }  t\in I_{t_0,\mathbf{x}_0};
    $$
    \item[$\bullet$] {\it uniformly stable} if, for all $\epsilon >0$, there exists $\delta=\delta(\epsilon) \in(0,r)$  such that for all $t_0\in \mathbb{R}^+$,
    $$
    \|\mathbf{x}_0\|<\delta \quad \text{ implies that }\quad \|\mathbf{x}(t,t_0,\mathbf{x}_0)\|< \epsilon \text{ for all } t\in I_{t_0,\mathbf{x}_0}.
    $$
    \end{itemize}
    \end{defn}

\begin{rem}\label{rem:stability yields global exist. of sol. starting from x_0<r}
  In the case where the trivial solution $\mathbf{x}=\mathbf{0}$ of the system~\eqref{eq:stability-dyn-system} is stable, for every $t_0\ge0$ fixed, observe that for $\epsilon=r$, there exists $\delta=\delta(\epsilon,t_0)\in (0,r)$ such that for all $\mathbf{x}_0 \in B_{\mathbb{R}^n}(\mathbf{0},\delta)$,
  $$
  \|\mathbf{x}(t,t_0,\mathbf{x}_0)\|< \epsilon \text{ for all }  t\in I_{t_0,\mathbf{x}_0}.
  $$
  Using Theorem~\ref{thm:3-alternatives of max-interval-existence} and~{(\rm H$_{\Omega}$)} and~{\rm(H$_r$)}, we deduce that $I_{t_0,\mathbf{x}_0}=[t_0,\infty)$. Furthermore, observe that if the stability of the trivial solution $\mathbf{x}=\mathbf{0}$ is uniform, then $\delta$ do not depend on $t_0$.
\end{rem}

The following definition is a notion of {\it asymptotic stability}. This concerns the behavior of solutions as $t\to \infty$.
\begin{defn}
The trivial solution $\mathbf{x}=\mathbf{0}$ of the system~\eqref{eq:stability-dyn-system} is said to be
\begin{itemize}
\item[$\bullet$] {\it asymptotically stable} if it is {\it stable} and, for every $t_0\in \mathbb{R}^+$, there exists $\delta=\delta(t_0)\in(0,r)$  such that, for all $\mathbf{x}_0 \in B_{\mathbb{R}^n}(\mathbf{0},\delta)$ and $\epsilon>0$, there exists $\sigma=\sigma(t_0,\mathbf{x}_0,\epsilon) >0$ such that
    $$
  \|\mathbf{x}(t,t_0,\mathbf{x}_0)\|< \epsilon  \text{ for all } t\in [t_0+\sigma,\infty) \cap I_{t_0,\mathbf{x}_0};
    $$

    \item[$\bullet$] {\it uniformly asymptotically stable} if it is {\it uniformly stable} and  there exists $\delta\in(0,r)$  such that, for every $\epsilon>0$, there exists $\sigma=\sigma(\epsilon)>0$ such that, for all $t_0 \in \mathbb{R}^+$ and $\mathbf{x}_0 \in B_{\mathbb{R}^n}(\mathbf{0},\delta)$,
    $$
\|\mathbf{x}(t,t_0,\mathbf{x}_0)\|<\epsilon \text{ for all } t\in [t_0+\sigma,\infty) \cap I_{t_0,\mathbf{x}_0}.
    $$
  \end{itemize}
\end{defn}

In the following example, we compare the stability properties of the trivial solution of a linear Stieltjes dynamical system to the ones in the classical case, observing the change of the stability properties depending on the sets $C_g$ and $D_g$. The resolution of linear Stieltjes differential equations has been studied in the literature, see for instance~\cite{FMarTo-OnFirstandSec,FP,ThesisFL,Mar,MarThesis}.
\begin{exm}\label{exm:changeof-stability-properties-linear-Stieltjes-dyn-system}
  Let us consider the linear Stieltjes dynamical system
\begin{equation}\label{eq:changeof-stability-properties-linear-Stieltjes-dyn-system}
\begin{aligned}
   x_g'(t) &= c x(t)\quad \text{for $g$-almost all $t\ge t_0\ge0$,}\\
   x(t_0)&=x_0\in\mathbb{R},
   \end{aligned}
\end{equation}
 for $c\in \mathbb{R}$. In the classical case of derivation where $g\equiv \operatorname{id}_\mathbb{R}$, for $c>0$, the equilibrium $x=0$ is not stable given that the solutions have the form $x_0 e^{c(t-t_0)}$ and they are not bounded on $[t_0,\infty)$.  Thus, for every $\epsilon>0$ and $t_0\ge0$, there is no $\delta>0$ such that
    $$
    |x_0|<\delta \quad \text{ implies that }\quad |x_0e^{c(t-t_0)}|< \epsilon \text{ for all } t\ge t_0.
    $$
However, for $c<0$, the equilibrium $x=0$ is asymptotically stable since stability holds and given that the solutions $x_0 e^{c(t-t_0)}\to 0$ for all $x_0 \in \mathbb{R}$. In the case where $c=0$, every constant $z \in \mathbb{R}$ is a uniformly stable equilibrium.

Now, let us reconsider the dynamical system~\eqref{eq:changeof-stability-properties-linear-Stieltjes-dyn-system}, where $g:\mathbb{R}\to\mathbb{R}$ is defined by $g(t)=t$ for $t\le 1$ and $g(t)=1$ for $t\ge 1$. Thus, the solutions of the problem~\eqref{eq:changeof-stability-properties-linear-Stieltjes-dyn-system} have the form $x_0 e^{c(g(t)-g(t_0))}$, $t\ge t_0$  for every $x_0\in\mathbb{R}$ and $t_0\ge 0$. Now, we show that the stability properties of the trivial solution $x=0$ of the system~\eqref{eq:changeof-stability-properties-linear-Stieltjes-dyn-system} differ from the classical case, for $c>0$ and $c<0$.  For $c>0$, we deduce the stability of the trivial solution $x=0$. Indeed, for all $\epsilon>0$ and $t_0\ge0$, there exists $\delta=\frac{\epsilon}{e^{c(g(1)-g(t_0))}}$ such that
$$
|x_0|<\delta \quad \text{ implies that }\quad |x_0 e^{c(g(t)-g(t_0))}|< \epsilon \text{ for all } t\ge t_0.
$$
Whereas, for $c<0$ the equilibrium $x=0$ is merely uniformly stable compared to the classical case, since for all $\epsilon>0$, there exists $\delta=\epsilon>0$ such that
$$
|x_0|<\delta \quad \text{ implies that }\quad |x_0 e^{c(g(t)-g(t_0))}|< \epsilon \text{ for all } t\ge t_0.
$$
Observe that asymptotic stability does not hold  for any $c\in \mathbb{R}$, since
$$
x_0 e^{c(g(t)-g(t_0))}\to x_0 e^{c(g(1)-g(t_0))}\nrightarrow 0,
$$
as $t\to \infty$ for all $t_0\ge0$ and $x_0 \in \mathbb{R}^*$.

Next, we change slightly the dynamical system~\eqref{eq:changeof-stability-properties-linear-Stieltjes-dyn-system}, to incorporate jumps. Let us consider the derivator $g_1(t)=t+\sum_{n\in\mathbb{N}} \chi_{[n,\infty)}(t)$ for all $t\in \mathbb{R}$, and  reconsider the dynamical system with $g_1$ instead:
 \begin{equation}\label{eq:Dg2-changeof-stability-properties-linear-Stieltjes-dyn-system}
 \begin{aligned}
   x_{g_1}'(t) &= \begin{dcases}
                c x(t) & \mbox{if } t\ge t_0\ge0,\text{ with } t\notin D_{g_1}=\mathbb{N}, \\
                \nu x(t),& \mbox{if } t\ge t_0\ge0,\text{ with } t\in D_{g_1},
              \end{dcases}
   \end{aligned}
\end{equation}
where $c,\nu\in \mathbb{R}$ with $\nu\in(-1,\infty)\setminus\{0\}$. Again the solution of~\eqref{eq:Dg2-changeof-stability-properties-linear-Stieltjes-dyn-system} satisfying $x(t_0)=x_0$ has the form  $x(t)=x_0 e_{c_*}(t,t_0)=x_0e^{\int_{[t_0,t)}c_*(s) \,d\mu_{g_1}(s)}$, where
$$c_*(t)=\begin{dcases}
                                    c & \mbox{if }  t\in[t_0,\infty)\setminus D_{g_1}, \\
                                   \log(1+\nu) & \mbox{if }  t\in[t_0,\infty)\cap D_{g_1}.
                                  \end{dcases}
$$
In Figure~\ref{fig:stability-Change-with-Dg}, we can observe different patterns depending on the values of $c$ and~$\nu$. This implies that the presence of discontinuities can both destabilize and restore the stability properties of a dynamical system.
\begin{figure}[H]
    \begin{subfigure}{.49\linewidth}
        \includegraphics[width=\linewidth]{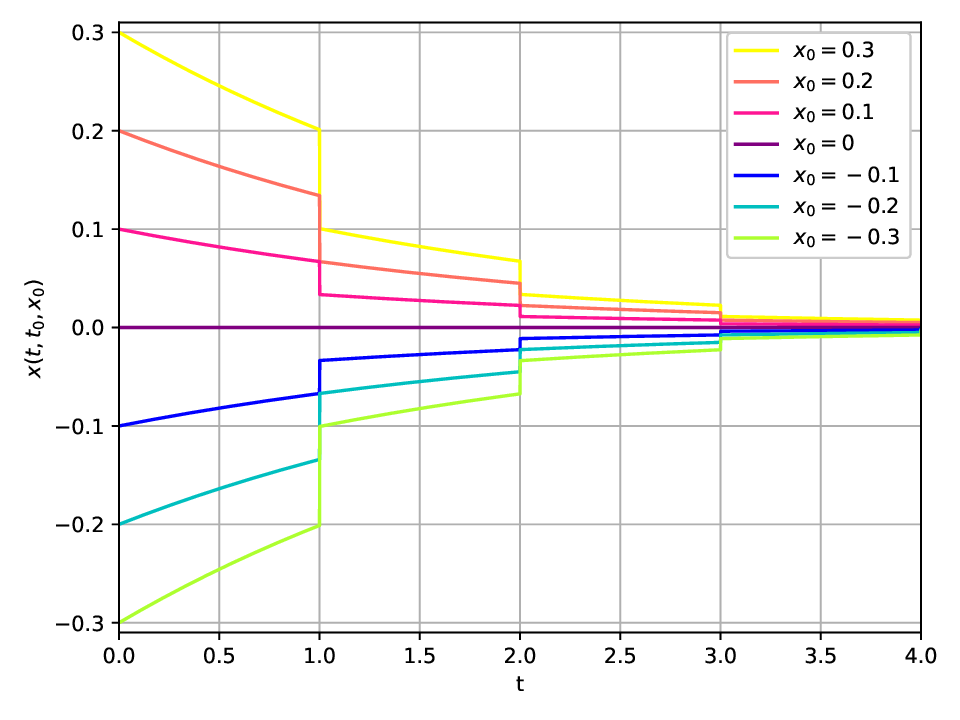}
        \caption{The case of $c=-0.4$, and $\nu=-0.5$.}
    \end{subfigure}\hfill
    \begin{subfigure}{.49\linewidth}
        \includegraphics[width=\linewidth]{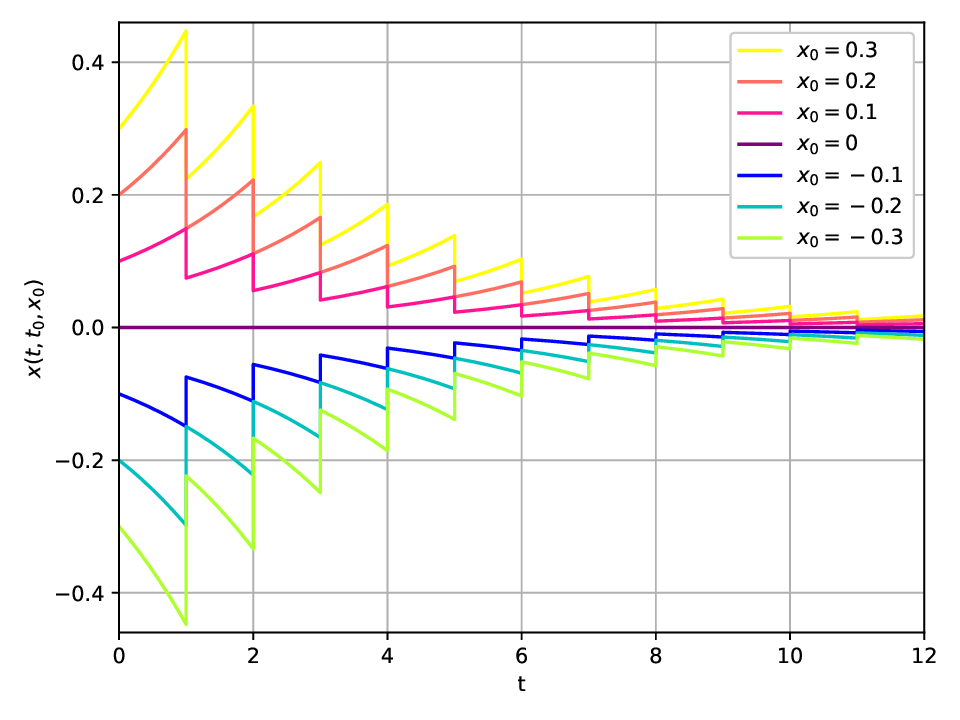}
        \caption{The case of $c=0.4$, and $\nu=-0.5$.}
    \end{subfigure}
        \begin{subfigure}{.49\linewidth}
        \includegraphics[width=\linewidth]{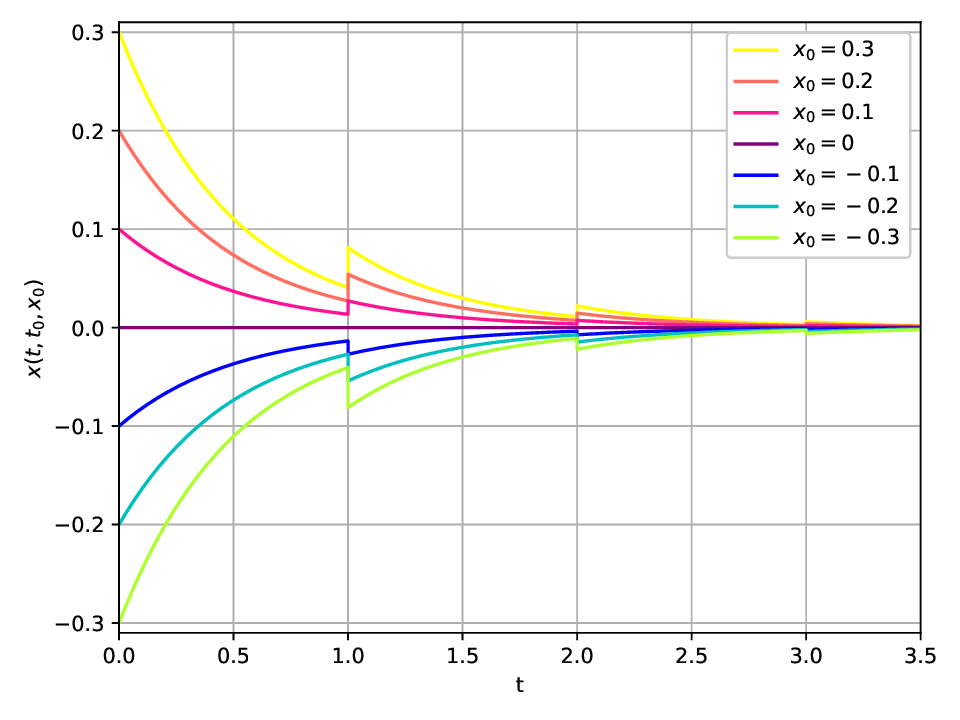}
        \caption{The case of $c=-2$, and $\nu=1$.}
    \end{subfigure}\hfill
    \begin{subfigure}{.49\linewidth}
        \includegraphics[width=\linewidth]{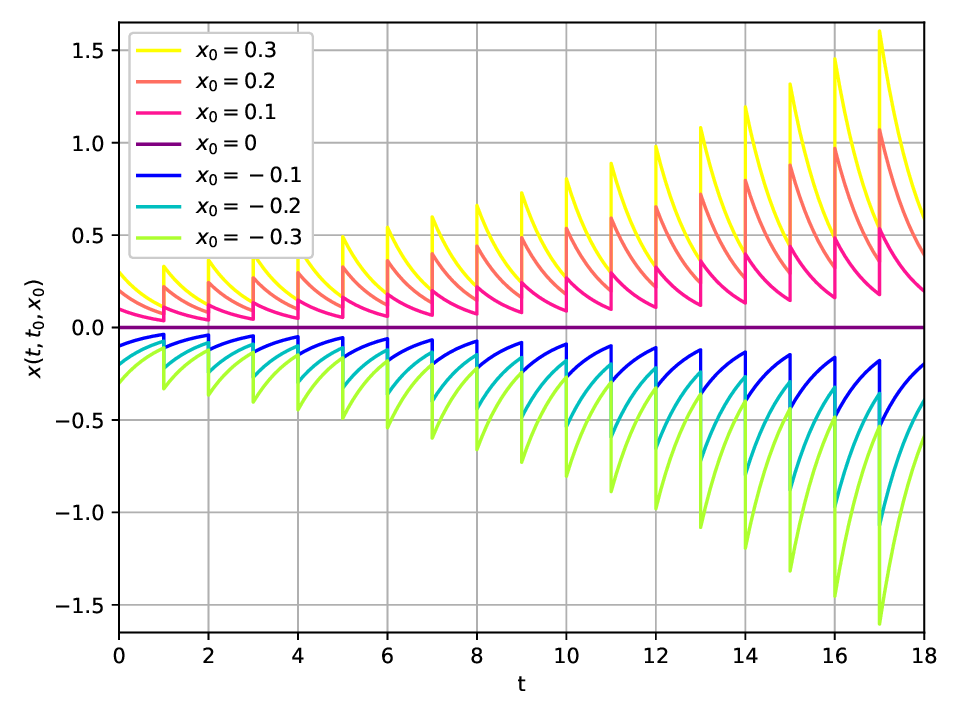}
        \caption{The case of $c=-1$, and $\nu=2$.}
    \end{subfigure}
        \caption{Behaviour of solutions in a neighborhood of $x=0$. }
    \label{fig:stability-Change-with-Dg}
\end{figure}
\end{exm}
\subsection{Stability results based on Lyapunov's function}
In the classical case where $g\equiv \operatorname{id}_\mathbb{R}$, Lyapunov's method is used to study the behavior of a trajectory of the system in a neighborhood of the trivial solution $\mathbf{x}=\mathbf{0}$, by means of a function $V$ depending on time and state; known as a {\it Lyapunov  function}. This function can be understood as an energy representation of the system~\eqref{eq:stability-dyn-system}, since in numerous applications, the function considered is the total energy of the system~\eqref{eq:stability-dyn-system} through time, see for instance~\cite{BL} for an example of an energy-based Lyapunov function for physical systems. In our context, the derivator $g$ takes into account the relevance of each moment during the process by means of the changes of the slopes of $g$ accordingly. Put differently, $g$ amplifies an alternative measurement for time, which may differ from the linear timeline typically used in the classical case where $g\equiv \operatorname{id}_\mathbb{R}$, see for instance the works~\cite{PM,PM2,PM3,PMM,MEF1,MEF2} where $g$ represents the life cycle of some populations, also we refer to~\cite{AFNT,FP} for more applications. Nevertheless, in the context of Stieltjes differentiation,   we still can rely on Lyapunov's function, particularly,  based on the $g$-derivative of its composition with maximal solutions of the system~\eqref{eq:stability-dyn-system} under consideration. This $g$-derivative of the composition will then permit a better  understanding of how the energy of the system~\eqref{eq:stability-dyn-system} changes, but with respect to this new observed time described by $g$. More precisely, it will describe how the energy of the system varies in response to the variation of this new "curved" scale of time.

Based on the definition of the Stieltjes derivative in Definition~\ref{df:g derivative}, we define the partial Stieltjes derivative as follows.
\begin{defn}\label{df:partial g-derivative}
Given a function $V:\mathbb{R}^+\times B_{\mathbb{R}^n}(\mathbf{0},r_0)\to \mathbb{R}$ and $t, x_1,\dots,x_n$ its arguments. The partial $g$-derivative of $V$ with respect to  the first argument at a point $(t,\mathbf{x})\in(\mathbb{R}^+\setminus C_g)\times  B_{\mathbb{R}^n}(\mathbf{0},r_0)$ is defined as:
$$
\frac{\partial V}{\partial_g t}(t,\mathbf{x})=
\begin{dcases}
\lim\limits_{s\to t} \frac{V(s,\mathbf{x})-V(t,\mathbf{x})}{g(s)-g(t)}, &\text{if } t \in \mathbb{R}^+ \setminus D_g,  \\
\frac{V(t^+,\mathbf{x})-V(t,\mathbf{x})}{g(t^+)-g(t)}, &\text{if } t \in \mathbb{R}^+ \cap D_g,
\end{dcases}
$$
provided that the limits exist.
\end{defn}

Combining Proposition~\ref{prop:g derivative of composition} and Definition~\ref{df:partial g-derivative}, we obtain the following technical proposition. The proof involves a formula related to the $g$-derivative of the composition involving a function with two variables. Formulae of this fashion were stated without proof in~\cite[Lemma~11]{SS} for $t\notin D_g$. In the next proposition, we derive formulae in the case where $D_g$ and $N_g$ are discrete.
\begin{prop}\label{prop:total g-derivative}
Let $g: \mathbb{R} \to \mathbb{R}$ be a left-continuous and nondecreasing function such that $D_g$ and $N_g$ are discrete. Given a function $V:\mathbb{R}^+\times B_{\mathbb{R}^n}(\mathbf{0},r_0) \to \mathbb{R}$ satisfying the following assumptions:
\begin{enumerate}[$(1)$]
  \item $V(\cdot,\mathbf{u})$ is $g$-differentiable on $\mathbb{R}^+\setminus (D_g \cup C_g)$  for all $\mathbf{u}\in B_{\mathbb{R}^n}(\mathbf{0},r_0)$;
  \item $V(t,\cdot) \in C^1(B_{\mathbb{R}^n}(\mathbf{0},r_0),\mathbb{R})$ for all $t\ge 0$;
  \item $\frac{\partial V}{\partial_g t}(\cdot,\mathbf{u})$ is continuous on $\mathbb{R}^+\setminus (D_g \cup C_g)$ for all $\mathbf{u}\in B_{\mathbb{R}^n}(\mathbf{0},r_0)$;
  \item for $(t_0,\mathbf{x}_0)\in \mathbb{R}^+\times  B_{\mathbb{R}^n}(\mathbf{0},r)$, $V(\cdot,\mathbf{x}(\cdot))\in\mathcal{AC}_{g,loc}(I_{\mathbf{x}},\mathbb{R})$ for every solution $\mathbf{x} :I_{\mathbf{x}} \to B_{\mathbb{R}^n}(\mathbf{0},r_0)$ of the system~\eqref{eq:stability-dyn-system}.
\end{enumerate}
Then, for all $[a,b]\subset I_{\mathbf{x}}$, we obtain for $g$-almost all $t\in [a,b]\setminus (D_g \cup C_g)$ that
\begin{equation}\label{eq:total g-derivative of V w.r. to system}
V_g'(t,\mathbf{x}(t))=
 \frac{\partial V}{\partial_g t}(t,\mathbf{x}(t))+ \sum_{i=1}^{n}\frac{\partial V}{\partial x_i}(t,\mathbf{x}(t)) f_i(t,\mathbf{x}(t)).
\end{equation}
Moreover, if $t\in [a,b)\cap D_g$, then
 \begin{equation}\label{eq:total g-derivative of V w.r. to system in D_g}
V_g'(t,\mathbf{x}(t))=\frac{V(t^+,\mathbf{x}(t)+ \mu_g(\{t\})\mathbf{f}(t,\mathbf{x}(t)))-V(t,\mathbf{x}(t))}{g(t^+)-g(t)}.
\end{equation}
\end{prop}
\begin{proof}
For $(t_0,\mathbf{x}_0)\in \mathbb{R}^+ \times B_{\mathbb{R}^n}(\mathbf{0},r)$, let us consider  a solution $\mathbf{x} :I_{\mathbf{x}} \to B_{\mathbb{R}^n}(\mathbf{0},r_0)$ of the system~\eqref{eq:stability-dyn-system}. Thus, the composition $V(\cdot,\mathbf{x}(\cdot))\in\mathcal{AC}_{g}(I_{\mathbf{x}},\mathbb{R})$. Let $[a,b]\subset I_{\mathbf{x}}$. For $g$-almost every $t \in [a,b] \setminus (D_g \cup C_g)$:
\begin{align*}
V_g'(t,\mathbf{x}(t))&= \lim\limits_{s\to t} \frac{V(s,\mathbf{x}(s))-V(t,\mathbf{x}(t))}{g(s)-g(t)} \\
 & =  \lim\limits_{s\to t} \frac{V(s,\mathbf{x}(s))-V(t,\mathbf{x}(s))}{g(s)-g(t)}+\frac{V(t,\mathbf{x}(s))-V(t,\mathbf{x}(t))}{g(s)-g(t)}\\
  & =  \lim\limits_{s\to t} \frac{V(s,\mathbf{x}(s))-V(t,\mathbf{x}(s))}{g(s)-g(t)}\\
  &\quad+\sum_{i=1}^{n}\Big( \frac{V(t,(x_1(t),\dots,x_{i-1}(t),x_i(s),\dots,x_n(s)))}{g(s)-g(t)}\\
&\qquad\quad-\frac{V(t,(x_1(t),\dots,x_i(t),x_{i+1}(s),\dots,x_n(s)))}{g(s)-g(t)}\Big).
\end{align*}
For $s$ sufficiently close to $t$, and since $D_g$ and $N_g$ are discrete, $g$ is continuous and increasing on the interval with endpoint points $s$ and $t$. By applying Corollary~\ref{cor:g-MeanValueTheorem} to the function $V(\cdot,\mathbf{x}(s))$, we obtain that there exists $c$ between $s$ and $t$ such that
$$
\frac{V(s,\mathbf{x}(s))-V(t,\mathbf{x}(s))}{g(s)-g(t)}= \frac{\partial V}{\partial_g t}(c,\mathbf{x}(s)).
$$
As $s\to t$, $c\to t$, and using Condition~(3) we obtain
$$
\frac{V(s,\mathbf{x}(s))-V(t,\mathbf{x}(s))}{g(s)-g(t)}= \frac{\partial V}{\partial_g t}(c,\mathbf{x}(s))\to \frac{\partial V}{\partial_g t}(t,\mathbf{x}(t)).
$$
Therefore,
\begin{align*}
V_g'(t,\mathbf{x}(t)) & =  \frac{\partial V}{\partial_g t}(t,\mathbf{x}(t))+ \sum_{i=1}^{n}\frac{\partial V}{\partial x_i}(t,\mathbf{x}(t))(x_i)_g'(t)\\
  &=  \frac{\partial V}{\partial_g t}(t,\mathbf{x}(t))+ \sum_{i=1}^{n}\frac{\partial V}{\partial x_i}(t,\mathbf{x}(t))f_i(t,\mathbf{x}(t)).
\end{align*}
For $t \in[a,b)\cap D_g$, we obtain immediately that
\begin{align*}
V_g'(t,\mathbf{x}(t))&=\frac{V(t^+,\mathbf{x}(t^+))-V(t,\mathbf{x}(t))}{g(t^+)-g(t)}\\
  &=\frac{V(t^+,\mathbf{x}(t)+\mu_g(\{t\})\mathbf{f}(t,\mathbf{x}(t)))-V(t,\mathbf{x}(t))}{g(t^+)-g(t)}.
 \end{align*}
\end{proof}

In order to establish sufficient conditions for different types of stability of the trivial solution $\mathbf{x}=\mathbf{0}$ of the system~\eqref{eq:stability-dyn-system} as defined in Subsection~4.1, we introduce specific sets of functions.
\begin{defn}\label{df:class V^1}
A function $V:\mathbb{R}^+\times B_{\mathbb{R}^n}(\mathbf{0},r_0) \to \mathbb{R}$ is said to belong  to class $\mathcal{V}^g_1$ if it satisfies the following conditions:
\begin{enumerate}
\item $V(t,\cdot)$ is continuous for all $t\ge 0$;
\item $V(\cdot,\mathbf{x}(\cdot))\in\mathcal{AC}_{g,loc}(I_{t_0,\mathbf{x}_0},\mathbb{R})$ for every function
$\mathbf{x}:I_{t_0,\mathbf{x}_0}\to B_{\mathbb{R}^n}(\mathbf{0},r_0)$ of $\mathcal{AC}_{g,loc}(I_{t_0,\mathbf{x}_0},\mathbb{R}^n)$ maximal solution of the system~\eqref{eq:stability-dyn-system};
\item $V(t,\mathbf{0})=0$ for all $t\ge0$.
\end{enumerate}
\end{defn}
\begin{defn}
 A  function $\varphi:\mathbb{R}^+\to \mathbb{R}^+$ belongs to the class $\mathcal{K}$ if it fulfills the following assumptions:
  \begin{enumerate}
   \item $\varphi$ is continuous;
   \item $\varphi(0)=0$;
   \item $\varphi$ is increasing.
  \end{enumerate}
\end{defn}
Now, we state the first stability result.
\begin{thm}\label{thm:stability+uniform stability}
Assume that Conditions~{\rm(H$_{\Omega}$)},~{\rm(H$_{\mathbf{f}}$)} and~{\rm(H$_r$)} hold.
 If there  exist functions $V \in \mathcal{V}_1^g$ and $a \in \mathcal{K}$ such that
 \begin{enumerate}
\item[{{\rm(a)}}] $a(\|\mathbf{u}\|)\le V(t,\mathbf{u})$, for all $(t,\mathbf{u})\in \mathbb{R}^+\times B_{\mathbb{R}^n}(\mathbf{0},r_0)$;
\item[{{\rm(b)}}] for every $(t_0,\mathbf{x}_0) \in \mathbb{R}^+\times B_{\mathbb{R}^n}(\mathbf{0},r)$, if $\mathbf{x}:I_{t_0,\mathbf{x}_0}\to B_{\mathbb{R}^n}(\mathbf{0},r_0)$ is a maximal solution of the system~\eqref{eq:stability-dyn-system}, then $V_g'(t,\mathbf{x}(t)) \le 0$ for $g$-almost all $t\in I_{t_0,\mathbf{x}_0}$.
\end{enumerate}
Then, the trivial solution of the system~\eqref{eq:stability-dyn-system} is
 \begin{enumerate}
\item[{{\rm(i)}}]{\it stable}.
\item[{{\rm(ii)}}]{\it uniformly stable} if there exists $b \in \mathcal{K}$ such that
    \begin{equation}\label{eq:V is decrescent}
    V(t,\mathbf{u})\le b(\|\mathbf{u}\|), \quad \text{ for all } (t,\mathbf{u})\in \mathbb{R}^+\times B_{\mathbb{R}^n}(\mathbf{0},r_0).
    \end{equation}
    \end{enumerate}
\end{thm}

\begin{proof}
{{\rm(i)}}\quad Since $V\in \mathcal{V}_1^g$, then, for all $\epsilon>0$ and $t_0\in \mathbb{R}^+ $, there exists $\delta=\delta(\epsilon, t_0) \in(0,r)$ such that
      $$
      \sup_{\|\mathbf{u}\|<\delta}V(t_0,\mathbf{u})<a(\epsilon).
      $$
Let $\mathbf{x}:I_{t_0,\mathbf{x}_0}\to B_{\mathbb{R}^n}(\mathbf{0},r_0)$ be a maximal solution of the system~\eqref{eq:stability-dyn-system} and $\|\mathbf{x}_0\|<\delta$. It follows from Conditions~{{\rm(a)}} and~{{\rm(b)}}  that
$$
a(\|\mathbf{x}(t)\|)\le V(t,\mathbf{x}(t)) \le V(t_0,\mathbf{x}_0) < a(\epsilon).
$$
Thus, for all $t\in I_{t_0,\mathbf{x}_0}$, we have that
$$
\|\mathbf{x}(t)\|=\|\mathbf{x}(t,t_0,\mathbf{x}_0)\| < \epsilon.
$$
 Therefore, the trivial solution $\mathbf{x}=\mathbf{0}$ is stable.\medskip

{{\rm(ii)}}\quad  Arguing as in~{{\rm(i)}}, we can choose a $\delta=\delta(\epsilon) \in(0,r)$ independent of $t_0$ such that
$$
 b(\delta)<a(\epsilon).
$$
Thus, using~\eqref{eq:V is decrescent}, we obtain for all $t\in I_{t_0,\mathbf{x}_0}$,
$$
a(\|\mathbf{x}(t)\|)\le V(t,\mathbf{x}(t)) \le V(t_0,\mathbf{x}_0)\le b(\|\mathbf{x}_0\|)< b(\delta)<a(\epsilon).
$$
This yields that the trivial solution $\mathbf{x}=\mathbf{0}$ is uniformly stable.
\end{proof}
In the next theorem, we impose additional assumptions which will permit to insure the asymptotical stability of the trivial solution $\mathbf{x}=\mathbf{0}$ to the system~\eqref{eq:stability-dyn-system}.

\begin{thm}\label{thm:asymptotic stability}
Assume that Conditions~{\rm(H$_{\Omega}$)},~{\rm(H$_{\mathbf{f}}$)} and~{\rm(H$_r$)} hold. Let $V \in \mathcal{V}_1^g$, $a, b \in \mathcal{K}$, $\phi : \mathbb{R}^+ \to \mathbb{R}^+$ continuous, and  a $g$-measurable function $w:\mathbb{R}^+\to \mathbb{R}^+$ be such that
 \begin{enumerate}
\item[{{\rm(a)}}] $a(\|\mathbf{u}\|) \le V(t,\mathbf{u})$ for every $(t,\mathbf{u})\in\mathbb{R}^+\times B_{\mathbb{R}^n}(\mathbf{0},r_0)$;
\item[{{\rm(b)}}] $\phi(s)=0$ if and only if $s=0$;
\item[{{\rm(c)}}] for every $(t_0,\mathbf{x}_0) \in \mathbb{R}^+\times B_{\mathbb{R}^n}(\mathbf{0},r)$, the maximal solution $\mathbf{x}:I_{t_0,\mathbf{x}_0}\to B_{\mathbb{R}^n}(\mathbf{0},r_0)$ of the system~\eqref{eq:stability-dyn-system} satisfies
 $$
 V_g'(t,\mathbf{x}(t)) \le -w(t)\phi(\|\mathbf{x}(t)\|) \quad \text{for $g$-almost all } t\in I_{t_0,\mathbf{x}_0};
 $$
 \item[{{\rm(d)}}] $\inf_{t_0 \in \mathbb{R}^+}\lim_{t\to +\infty} \int_{[t_0,t_0+t)} w(s)\,d\mu_g(s) =\infty$.
 \end{enumerate}
If the trivial solution $\mathbf{x}=\mathbf{0}$ of the system~\eqref{eq:stability-dyn-system} is uniformly stable, then $\mathbf{x}=\mathbf{0}$ is {\rm asymptotically stable}.
\end{thm}

\begin{proof}
(i) The stability of the trivial solution $\mathbf{x}=\mathbf{0}$ holds from uniform stability. Let us choose  $\delta_0 \in(0, r)$ associated to  an $\epsilon_0 \le r$ given by the uniform stability.  Now, for a fixed $t_0 \ge 0$, let $\epsilon>0$. Again, by the uniform stability, there exists $\delta \in (0,\delta_0)$ such that, for all $\hat{t} \in \mathbb{R}^+$ and every $\hat{\mathbf{x}}_0$ satisfying $\|\hat{\mathbf{x}}_0\| < \delta$, one has
$$
\|\hat{\mathbf{x}}(t,\hat{t},\hat{\mathbf{x}}_0)\|< \epsilon  \text{ for all } t\in [\hat{t},\infty) \cap I_{\hat{t},\hat{\mathbf{x}}_0}.
$$
We denote
\begin{equation}\label{eq:proof:thm:asym-stab:M=inf phi}
M = \inf_{s \in [\delta,r_0)}|\phi(s)|.
\end{equation}
By Condition~{{\rm(b)}}, observe that $M>0$.

Let $\mathbf{x}_0 \in B_{\mathbb{R}^n}(\mathbf{0},\delta_0)$. Since
$$
\lim_{t\to +\infty} \int_{[t_0,t_0+t)} w(s)\,d\mu_g(s) =\infty,
$$
we can choose $\sigma > 0$ such that
$$
\int_{[t_0,t_0+\sigma)} w(s)\,d\mu_g(s) > \frac{V(t_0,\mathbf{x}_0)}{M}.
$$
Let $\mathbf{x} : I_{t_0,\mathbf{x}_0} \to B_{\mathbb{R}^n}(\mathbf{0},r_0)$ be a maximal solution of~\eqref{eq:stability-dyn-system}. Using Remark~\ref{rem:stability yields global exist. of sol. starting from x_0<r}, notice that $I_{t_0,\mathbf{x}_0}=[t_0,\infty)$. Now, if there exists $\hat{t} \in [t_0,t_0+\sigma]$ such that $\|\mathbf{x}(\hat{t})\| < \delta$, then, by the uniform stability
 $$
 \|\hat{\mathbf{x}}(t)\| < \epsilon \quad \text{for all } t \in [\hat{t},\infty)\cap I_{\hat{t},\mathbf{x}(\hat{t})},
 $$
 where $\hat{\mathbf{x}}:  I_{\hat{t},\mathbf{x}(\hat{t})} \to B_{\mathbb{R}^n}(\mathbf{0},r_0)$ is the maximal solution of~\eqref{eq:stability-dyn-system} satisfying the initial condition $\hat{\mathbf{x}}(\hat{t}) = \mathbf{x}(\hat{t})$. By the uniqueness of the maximal solution, one has
 $$
 \omega(t_0,\mathbf{x}_0) = \omega(\hat{t},\mathbf{x}(\hat{t}))=\infty \quad \text{and}\quad \mathbf{x}(t) = \hat{\mathbf{x}}(t) \ \text{for all } t\in [\hat{t},\infty).
 $$
 Hence,
 $$
 \|\mathbf{x}(t)\| < \epsilon \quad \text{for all } t \in [t_0+\sigma,\infty)\cap I_{t_0,\mathbf{x}_0}.
 $$
 On the other hand, if $\|\mathbf{x}(t)\| \ge \delta$ for all $t \in [t_0,t_0+\sigma]$, then using Condition~{\rm(a)}, Theorem~\ref{th:Fund-Th}, and~\eqref{eq:proof:thm:asym-stab:M=inf phi}, we obtain
\begin{align*}
 a(\|\mathbf{x}(t_0+\sigma)\|) &\le V(t_0+\sigma,\mathbf{x}(t_0+\sigma))
 \\
 &=V(t_0,\mathbf{x}(t_0,t_0,\mathbf{x}_0))+\int_{[t_0,t_0+\sigma)}V_g'(s,\mathbf{x}(s))\,d\mu_g(s)
 \\
 & \le V(t_0,\mathbf{x}_0) -\int_{[t_0,t_0+\sigma)} w(s) \phi(\|\mathbf{x}(s)\|)\,d\mu_g(s)
 \\
 & \le V(t_0,\mathbf{x}_0) - M\int_{[t_0,t_0+\sigma)} w(s)\,d\mu_g(s)
 \\
 &< 0.
\end{align*}
 This is a contradiction. Therefore, $\mathbf{x}=\mathbf{0}$ is asymptotically stable.
\end{proof}

In the example below, we provide an application of Theorem~\ref{thm:asymptotic stability}.
\begin{exm}\label{exm:asym-stability-1}
  Let us consider the Stieltjes dynamical system
\begin{equation}\label{eq:exm:asym-stability-1}
   x_g'(t) = f(t,x(t))\quad \text{for $g$-almost all $t\ge t_0\ge0$}, \qquad
   x(t_0)=x_0\in\mathbb{R},
\end{equation}
with $g:\mathbb{R}\to \mathbb{R}$  defined by $g(t)=t$ for all $t\le 1$, and $g(t)=t+1$ for $t> 1$, and where $f:\mathbb{R}^+\times\mathbb{R} \to \mathbb{R}$ is a function defined by
$$
f(t,x)=\begin{dcases}
         -\frac{xt}{1+t^2} & \mbox{if } t\in \mathbb{R}^+\setminus D_g, \\
         \nu x & \mbox{if } t\in \mathbb{R}^+\cap D_g, \\
       \end{dcases}
$$
for some $\nu\in\mathbb{R}\setminus\{-1,0\}$. The function $f$ satisfies conditions of Theorem~\ref{thm:global-existence-with-linear-growth-condition}, thus the problem~\eqref{eq:exm:asym-stability-1} has a maximal solution $x: [t_0,+\infty)\to \mathbb{R}$ for every $(t_0,x_0) \in \mathbb{R}^+\times\mathbb{R}$. Observe that $x=0$ is an equilibrium of the dynamical system~\eqref{eq:exm:asym-stability-1}.

Let us define the function $V:\mathbb{R}^+\times \mathbb{R}\to \mathbb{R}$ for every $(t,x) \in\mathbb{R}^+\times \mathbb{R}$ by
$$
V(t,x)=\begin{dcases}
x^2  & \mbox{if } t\in [0,1],\\
 \frac{x^2}{(1+\nu)^2} & \mbox{if } t>1.
      \end{dcases}
$$
Clearly $V\in \mathcal{V}_g^1$, and $a(|u|)\le V(t,u)\le b(|u|)$ for all $(t,u)\in \mathbb{R}^+\times \mathbb{R}$, where $a,b\in \mathcal{K}$ are given by $a(s)=\min\left\{s^2,\frac{s^2}{(1+\nu)^2}\right\}$ and $b(s)=\max\left\{s^2,\frac{s^2}{(1+\nu)^2}\right\}$ for all $s\in \mathbb{R}^+$. In addition, for all $(t,x)\in \mathbb{R}^+\times \mathbb{R}$:
\begin{equation*}
\begin{aligned}
\frac{\partial V}{\partial_g t}(t,x)&=\begin{dcases}
                                       0 & \mbox{if } t\in (0,1)\cup(1,\infty),\\
                                       \frac{\frac{x^2}{(1+\nu)^2}-x^2}{g(1^+)-g(1)} & \mbox{if } t=1,\\
                                     \end{dcases}\\
                                     &=\begin{dcases}
                                       0 & \mbox{if } t\in (0,1)\cup(1,\infty),\\
 \frac{1-(1+\nu)^2}{(1+\nu)^2}x^2 & \mbox{if } t=1.\\
                                       \end{dcases}
\end{aligned}
\end{equation*}
Thus, by means of Proposition~\ref{prop:total g-derivative}, for $t\in [t_0,\infty)\setminus D_g$, we obtain
\begin{equation*}
\begin{aligned}
       V_g'(t,x(t))&= \frac{\partial V}{\partial_g t}(t,x(t))+\frac{\partial V}{\partial x}(t,x(t)) f(t,x(t))\\
 &= \begin{dcases}
 -\frac{2t}{1+t^2} x(t)^2 & \mbox{if } t\in [t_0,\infty)\cap [0,1),\\
 -\frac{2t}{(1+t^2)(1+\nu)^2}x(t)^2 &   \mbox{if } t\in [t_0,\infty)\cap (1,\infty).
 \end{dcases}
  \end{aligned}
\end{equation*}
For $t\in [t_0,\infty)\cap D_g$, if $t_0 \le 1$, then $t=1$ and we have that
\begin{equation*}
\begin{aligned}
       V_g'(1,x(1))&=\frac{V(1^+,x(1^+))-V(1,x(1))}{g(1^+)-g(1)}\\
&=\frac{V(1^+,x(1)+ \mu_g(\{1\})f(1,x(1)))-V(1,x(1))}{g(1^+)-g(1)}\\
&=\frac{\frac{(1+\nu)^2x(1)^2}{(1+\nu)^2}-x(1)^2}{g(1^+)-g(1)}\\
 &= 0.
  \end{aligned}
\end{equation*}
 This implies that
$$
 V_g'(t,x(t))\le-\omega(t)\phi(|x(t)|), \text{ for every } t\ge t_0,
 $$
with $w:\mathbb{R}^+ \to\mathbb{R}^+$, defined by
$$
w(t)=\begin{dcases}
      \frac{2t}{1+t^2}  & \mbox{if } t\in[0,1), \\
       0 & \mbox{if } t=1,\\
      \frac{2t}{(1+t^2)(1+\nu)^2} & \mbox{if } t\in(1,\infty),
     \end{dcases}
$$
 and $\phi\in \mathcal{K}$, defined by $\phi(y)=y^2$ for all $y\in\mathbb{R}^+$. Moreover, for every $t_0\in \mathbb{R}^+$  and $t>1$, we have
\begin{align*}
\int_{[t_0,t_0+t)} w(s)\,d\mu_g(s)  & = \int_{[t_0,t_0+t)\cap([0,1)\cup \{1\}\cup (1,\infty))} w(s)\,d\mu_g(s)\\
 &\ge \frac{1}{(1+\nu)^2} \int_{[t_0,t_0+t)\cap (1,\infty)} \frac{2s}{1+s^2}\,ds \\
   & =  \frac{1}{(1+\nu)^2}\big(\log(1+(t_0+t)^2)-\sup\{\log(2),\log(1+t_0^2)\}\big).
\end{align*}
As, $\inf_{t_0 \in \mathbb{R}^+}\lim_{t\to +\infty} \log(1+(t_0+t)^2)-\sup\{\log(2),\log(1+t_0^2)\} =+\infty$, we obtain that
\begin{equation*}
\inf_{t_0 \in \mathbb{R}^+}\lim_{t\to +\infty} \int_{[t_0,t_0+t)} w(s)\,d\mu_g(s)=+\infty.
\end{equation*}
Thus, Condition~{{\rm(d)}} holds. Therefore, by means of Theorem~\ref{thm:asymptotic stability}, we deduce that $x=0$ is an asymptotically stable equilibrium. Figure~\ref{fig:Asymp-Behav-example1} illustrates the asymptotic behavior of solutions of the dynamical system~\eqref{eq:exm:asym-stability-1}.
\begin{figure}[H]
  \centering
  \includegraphics[height=7.5cm]{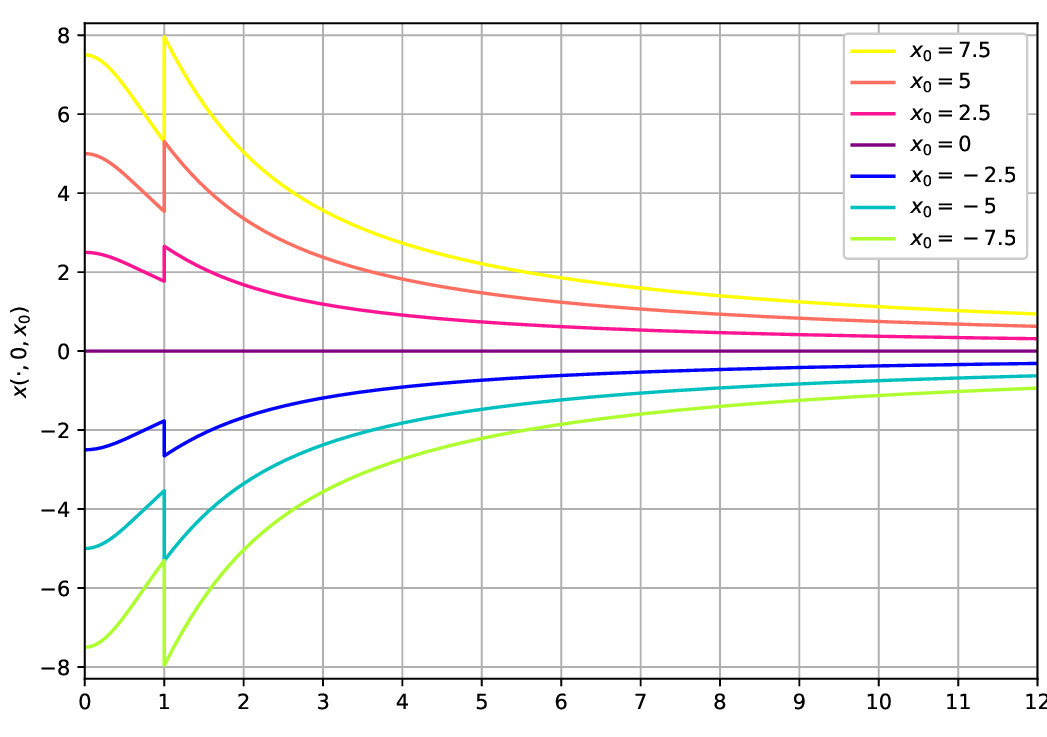}
  \caption{Asymptotic behavior of solutions of the dynamical system~\eqref{eq:exm:asym-stability-1} obtained with a time-discretization step-size of $10^{-3}$, with $\nu=0.5$.}\label{fig:Asymp-Behav-example1}
\end{figure}
\end{exm}

 In the following example, we reconsider the system~\eqref{eq:exm:asym-stability-1} in the case where the set $D_g$ is infinite.
 \begin{exm}\label{exm:asym-stability-2}
  Let us consider the dynamical system
\begin{equation}\label{eq:exm:asym-stability-2}
   x_g'(t) = f(t,x(t))\quad \text{for $g$-almost all $t\ge t_0\ge0$}, \qquad
   x(t_0)=x_0\in\mathbb{R},
\end{equation}
where $g:\mathbb{R}\to \mathbb{R}$ is a derivator such that $D_g=\{t_k\}_{k\in\mathbb{N}}\subset(0,+\infty)$, defined by
$$
g(t)=t+\sum_{k\in\mathbb{N}}\chi_{[t_k,+\infty)}(t)\quad \text{for all $t\in \mathbb{R}$},
$$
 and $f:\mathbb{R}^+\times\mathbb{R} \to \mathbb{R}$ is a function defined by
$$
f(t,x)=\begin{dcases}
         -\frac{xt}{1+t^2} & \mbox{if } t\in \mathbb{R}^+\setminus D_g, \\
         \nu x & \mbox{if } t\in \mathbb{R}^+\cap D_g, \\
       \end{dcases}
$$
for some $\nu\in[-2,-1)$. The function $f$ satisfies conditions of Theorem~\ref{thm:global-existence-with-linear-growth-condition}, thus the problem~\eqref{eq:exm:asym-stability-1} has a maximal solution $x: [t_0,+\infty)\to \mathbb{R}$ for every $(t_0,x_0) \in \mathbb{R}^+\times\mathbb{R}$. Observe that $x=0$ is an equilibrium of the dynamical system~\eqref{eq:exm:asym-stability-2}.

Let us define the function $V:\mathbb{R}^+\times \mathbb{R}\to \mathbb{R}$ for all $(t,x) \in\mathbb{R}^+\times \mathbb{R}$ by
$$
V(t,x)=x^2
$$
Clearly $V\in \mathcal{V}_g^1$, and $a(|u|)\le V(t,u)\le b(|u|)$ for all $(t,u)\in \mathbb{R}^+\times \mathbb{R}$, where $a\in \mathcal{K}$ is given by $a(s)=s^2$ and $b(s)=\frac{s^2}{(1+\nu)^2}$ for all $s\in \mathbb{R}^+$. In addition, for all $t\in [t_0,\infty)$
\begin{equation*}
\frac{\partial V}{\partial_g t}(t,x)=0.
\end{equation*}
Thus, by means of Proposition~\ref{prop:total g-derivative}, for $g$-almost every $t\in [t_0,\infty)\setminus D_g$, we obtain
\begin{equation*}
\begin{aligned}
       V_g'(t,x(t))&= \frac{\partial V}{\partial_g t}(t,x(t))+\frac{\partial V}{\partial x}(t,x(t)) f(t,x(t))\\
 &=  -\frac{2t}{1+t^2} x(t)^2.
  \end{aligned}
\end{equation*}
For $t_k\in [t_0,\infty)\cap D_g$, we have that
\begin{equation*}
\begin{aligned}
V_g'(t_k,x(t_k))&=\frac{V(t_k^+,x(t_k^+))-V(t_k,x(t_k))}{g(t_k^+)-g(t_k)}\\
&=\frac{V(t_k^+,x(t_k)+ \mu_g(\{t_k\})f(t_k,x(t_k)))-V(t_k,x(t_k))}{g(t_k^+)-g(t_k)}\\
       &= \frac{((1+\nu)x(t_k))^2-x(t_k)^2}{g(t_k^+)-g(t_k)}\\
       &=\nu(2+\nu)x(t_k)^2.
  \end{aligned}
\end{equation*}
 This implies that
$$
 V_g'(t,x(t))\le-\omega(t)\phi(|x(t)|), \text{ for $g$-almost every } t\ge t_0,
 $$
with $w:\mathbb{R}^+ \to\mathbb{R}^+$, defined by
$$
w(t)=\begin{dcases}
           \frac{2t}{1+t^2} & \mbox{if } t\in[0,t_1)\cup(t_k,t_{k+1}),\, k\in \mathbb{N},\\
       -\nu(2+\nu) & \mbox{if } t\in\{t_k\}_{k\in \mathbb{N}},
     \end{dcases}
$$
and $\phi\in \mathcal{K}$ defined by $\phi(y)=y^2$ for all $y\in\mathbb{R}^+$. For $t_0,t \in\mathbb{R}^+$, observe that $w$ satisfies:
$$
\begin{aligned}
\int_{[t_0,t_0+t)}w(s)\,d\mu_g(s)&=  \sum_{s\in [t_0,t_0+t)\cap D_g} -\nu(2+\nu) \mu_g(\{s\}) + \int_{t_0}^{t_0+t}\frac{2s}{1+s^2}\,ds\\
&= \sum_{s\in [t_0,t_0+t)\cap D_g} -\nu(2+\nu) \mu_g(\{s\}) + \log\Big(\frac{1+(t_0+t)^2}{1+t_0^2}\Big).
\end{aligned}
$$
Thus,
$$
\inf_{t_0 \in \mathbb{R}^+} \lim\limits_{t\to +\infty} \int_{[t_0,t_0+t)}w(s)\,d\mu_g(s)=+\infty.
$$
Consequently, Condition~{{\rm(d)}} holds. Therefore, by means of Theorem~\ref{thm:asymptotic stability}, we deduce that $x=0$ is an asymptotically stable equilibrium. Figure~\ref{fig:exm:asym-stability-2} illustrates the asymptotic behavior of solutions of the dynamical system~\eqref{eq:exm:asym-stability-2}.
\begin{figure}[H]
  \centering
  \includegraphics[height=7.5cm]{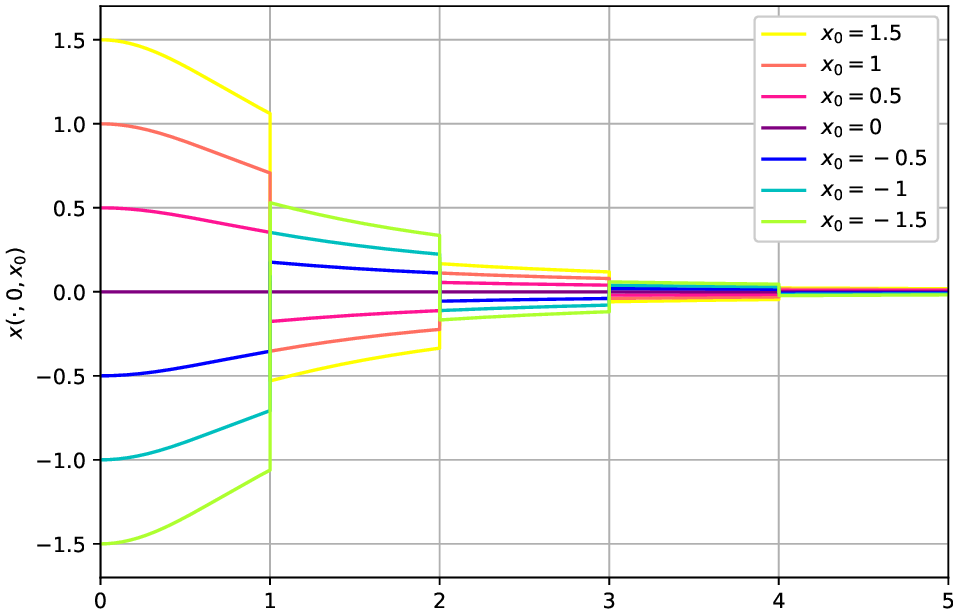}
  \caption{Asymptotic behavior of solutions of the dynamical system~\eqref{eq:exm:asym-stability-2} obtained with a time-discretization step-size of $10^{-3}$, with $\nu=-3/2$.}\label{fig:exm:asym-stability-2}
\end{figure}
\end{exm}

\begin{thm}\label{thm:uniform asymptotic stability}
Assume that Conditions~{\rm(H$_{\Omega}$)},~{\rm(H$_{\mathbf{f}}$)} and~{\rm(H$_r$)} hold. If there  exist $V \in \mathcal{V}_1^g$, $a, b \in \mathcal{K}$, $\phi : \mathbb{R}^+ \to \mathbb{R}^+$ continuous, and  a $g$-measurable function $w:\mathbb{R}^+\to \mathbb{R}^+$ such that
\begin{enumerate}
\item[{{\rm(a)}}] $a(\|\mathbf{u}\|) \le V(t,\mathbf{u}) \le b(\|\mathbf{u}\|)$ for every $(t,\mathbf{u})\in\mathbb{R}^+\times B_{\mathbb{R}^n}(\mathbf{0},r_0)$;
\item[{{\rm(b)}}] $\phi(s)=0$ if and only if $s=0$;
\item[{{\rm(c)}}] for every $(t_0,\mathbf{x}_0) \in \mathbb{R}^+\times B_{\mathbb{R}^n}(\mathbf{0},r)$, the maximal solution $\mathbf{x}:I_{t_0,\mathbf{x}_0}\to B_{\mathbb{R}^n}(\mathbf{0},r_0)$ of the system~\eqref{eq:stability-dyn-system} satisfies
 $$
 V_g'(t,\mathbf{x}(t)) \le -w(t)\phi(\|\mathbf{x}(t)\|) \quad \text{for $g$-almost all } t\in I_{t_0,\mathbf{x}_0}.
 $$
 \item[{{\rm(d)}}] $\lim_{t\to +\infty} \inf_{t_0 \in \mathbb{R}^+}\int_{[t_0,t_0+t)} w(s)\,d\mu_g(s) =+\infty$.
 \end{enumerate}
Then the trivial solution $\mathbf{x}=\mathbf{0}$ of the system~\eqref{eq:stability-dyn-system} is {\it uniformly asymptotically stable}.
\end{thm}
\begin{proof}
By~(ii) of Theorem~\ref{thm:stability+uniform stability}, the trivial solution $\mathbf{x}=\mathbf{0}$ is uniformly stable. Thus, let us choose  $\delta_0 \in(0, r)$ associated to  an $\epsilon_0 \le r$. Let $\epsilon > 0$. Again, by uniform stability, there exists   $\delta \in(0,r)$ such that, for all $\hat{t} \in \mathbb{R}^+$ and every $\hat{\mathbf{x}}_0$ such that $\|\hat{\mathbf{x}}_0\| < \delta$, one has
 $$
 \|\hat{\mathbf{x}}(t,\hat{t},\hat{\mathbf{x}}_0)\|< \epsilon  \text{ for all } t\in [\hat{t},\infty) \cap I_{\hat{t},\hat{\mathbf{x}}_0}.
 $$
 Let $M$ be as defined in~\eqref{eq:proof:thm:asym-stab:M=inf phi}.  Since
 $$
 \lim_{t\to +\infty}\inf_{t_0 \in \mathbb{R}^+} \int_{[t_0,t_0+t)} w(s)\,d\mu_g(s) =+\infty,
 $$
 we can choose $\sigma > 0$ such that
 $$
 \int_{[t_0,t_0+\sigma)} w(s)\,d\mu_g(s) > \frac{b(\delta_0)}{M} \quad \text{for all } t_0 \in \mathbb{R}^+.
 $$
 Let $(t_0,\mathbf{x}_0) \in \mathbb{R}^+\times B_{\mathbb{R}^n}(\mathbf{0},\delta_0)$ and $\mathbf{x} : I_{t_0,\mathbf{x}_0} \to B_{\mathbb{R}^n}(\mathbf{0},r_0)$ a maximal solution of~\eqref{eq:stability-dyn-system}.

 If there exists $\hat{t} \in [t_0,t_0+\sigma] \cap I_{t_0,\mathbf{x}_0}$ such that $\|\mathbf{x}(\hat{t})\| < \delta$, then
 $$
 \|\hat{\mathbf{x}}(t)\| < \epsilon \quad \text{for all } t \in [\hat{t},\infty)\cap I_{\hat{t},\mathbf{x}(\hat{t})},
 $$
 where $\hat{\mathbf{x}}:  I_{\hat{t},\mathbf{x}(\hat{t})}\to B_{\mathbb{R}^n}(\mathbf{0},r_0)$ is the maximal solution of~\eqref{eq:stability-dyn-system} satisfying the initial condition $\hat{\mathbf{x}}(\hat{t}) = \mathbf{x}(\hat{t})$. By the uniqueness of the maximal solution, one has
 $$
 \omega(t_0,\mathbf{x}_0) = \omega(\hat{t},\mathbf{x}(\hat{t})) \quad \text{and}\quad \mathbf{x}(t) = \hat{\mathbf{x}}(t) \ \text{for all } t\in [\hat{t},\infty)\cap I_{t_0,\mathbf{x}_0}.
 $$
 Hence,
 $$
 \|\mathbf{x}(t)\| < \epsilon \quad \text{for all } t \in [t_0+\sigma,\infty)\cap I_{t_0,\mathbf{x}_0}.
 $$
 On the other hand, if $\|\mathbf{x}(t)\| \ge \delta$ for all $t \in [t_0,t_0+\sigma]$, then using Condition~{\rm(a)}, Theorem~\ref{th:Fund-Th}, and~\eqref{eq:proof:thm:asym-stab:M=inf phi}, we obtain
\begin{align*}
 a(\|\mathbf{x}(t_0+\sigma)\|) &\le V(t_0+\sigma,\mathbf{x}(t_0+\sigma))
 \\
 &=V(t_0,\mathbf{x}(t_0,t_0,\mathbf{x}_0))+\int_{[t_0,t_0+\sigma)}V_g'(s,\mathbf{x}(s))\,d\mu_g(s)
 \\
   & \le V(t_0,\mathbf{x}_0) -\int_{[t_0,t_0+\sigma)} w(s) \phi(\|\mathbf{x}(s)\|)\,d\mu_g(s)
   \\
   & \le V(t_0,\mathbf{x}_0) - M\int_{[t_0,t_0+\sigma)} w(s)\,d\mu_g(s)
   \\
   & <b(\|\mathbf{x}_0\|) -b(\|\mathbf{x}_0\|)\\
   &= 0.
\end{align*}
 This is a contradiction. Hence, we conclude that $\mathbf{x}=\mathbf{0}$ is uniformly asymptotically stable.
\end{proof}

 In the next example, we provide an application of Theorem~\ref{thm:uniform asymptotic stability} for a system subject to impulses where the trivial solution $\mathbf{x}=\mathbf{0}$ is uniformly asymptotically stable.
\begin{exm}\label{exm:asym-stability-3}
 Let us consider the dynamical system
\begin{equation}\label{eq:exm:asym-stability-3}
   x_g'(t) = f(t,x(t))\quad \text{for $g$-almost all $t\ge t_0\ge0$}, \qquad
   x(t_0)=x_0\in\mathbb{R},
\end{equation}
where $g:\mathbb{R}\to \mathbb{R}$ is a derivator such that $D_g=\{t_k\}_{k\in\mathbb{N}}\subset(0,+\infty)$, defined by
$$
g(t)=t+\sum_{k\in\mathbb{N}}\chi_{[t_k,+\infty)}(t)\quad \text{for all $t\in \mathbb{R}$},
$$
 and $f:\mathbb{R}^+\times\mathbb{R} \to \mathbb{R}$ is a function defined by
$$
f(t,x)=\begin{dcases}
         -t \arctan(x) & \mbox{if } t\in \mathbb{R}^+\setminus D_g, \\
         \nu_k x & \mbox{if } t=t_k,\, k\in\mathbb{N} \\
       \end{dcases}
$$
where  $\{\nu_k\}_k\subset\mathbb{R}_+^*$ is a sequence satisfying
$$
\lim_{k\to \infty}\frac{1}{\prod_{i=1}^{k}(1+\nu_i)^2}=a_0>0.
$$
 The map $f$ satisfies conditions of Theorem~\ref{thm:global-existence-with-linear-growth-condition}, thus the problem~\eqref{eq:exm:asym-stability-3} has a maximal solution $x: [t_0,+\infty)\to \mathbb{R}$ for every $(t_0,x_0) \in \mathbb{R}^+\times\mathbb{R}$. Observe that $x=0$ is an equilibrium of the Stieltjes dynamical system~\eqref{eq:exm:asym-stability-3}.

Let us define the function $V:\mathbb{R}^+\times \mathbb{R}\to \mathbb{R}$ for all $(t,x) \in\mathbb{R}^+\times \mathbb{R}$ by
$$
V(t,x)=\begin{dcases}
x^2  & \mbox{if } t\in [0,t_1],\\
\frac{x^2}{\prod_{i=1}^{k}(1+\nu_i)^2} & \mbox{if } t\in (t_k,t_{k+1}],\, k\in\mathbb{N}.
      \end{dcases}
$$
Clearly $V\in \mathcal{V}_g^1$, and $a(|x|)\le V(t,x)\le b(|x|)$ for all $(t,x)\in \mathbb{R}^+\times \mathbb{R}$, where $a,b\in \mathcal{K}$ are functions defined by $a(s)=a_0s^2$ and $b(s)=s^2$ for all $s\in \mathbb{R}^+$. In addition, for all $(t,x)\in \mathbb{R}^+\times \mathbb{R}$:
\begin{equation*}
\begin{aligned}
\frac{\partial V}{\partial_g t}(t,x)&=\begin{dcases}
                                       0 &\mbox{if } t\in [0,t_1]\cup (t_k,t_{k+1}),\, k\in\mathbb{N},\\
                                       \frac{\frac{x^2}{\prod_{i=1}^{k}(1+\nu_i)^2}-\frac{x^2}{\prod_{i=1}^{k-1}(1+\nu_i)^2}}{g(t_k^+)-g(t_k)} & \mbox{if }t=t_k,\, k\in\mathbb{N},
                                     \end{dcases}\\
                                     &=\begin{dcases}
0  & \mbox{if } t\in [0,t_1)\cup (t_k,t_{k+1}),\, k\in\mathbb{N},\\
\frac{1-(1+\nu_k)^2}{\prod_{i=1}^{k}(1+\nu_i)^2}x^2
 & \mbox{if }t=t_k,\, k\in\mathbb{N}.
                                       \end{dcases}
\end{aligned}
\end{equation*}
For $g$-almost every $t\in [t_0,\infty)\setminus D_g$, we have that $t\in[0,t_1)$ or there exists $k\in\mathbb{N}$ such that $t\in(t_k,t_{k+1})$. Thus, by means of Proposition~\ref{prop:total g-derivative}, we obtain if $t\in[0,t_1)$:
\begin{equation*}
\begin{aligned}
       V_g'(t,x(t))&= \frac{\partial V}{\partial_g t}(t,x(t))+\frac{\partial V}{\partial x}(t,x(t)) f(t,x(t))\\
 &= -2x(t)t\arctan(x(t)) \\
      &\le     -2t \frac{x(t)^2}{1+x(t)^2},
  \end{aligned}
\end{equation*}
where the last inequality follows from the Mean Value Theorem. While if there exists $k\in\mathbb{N}$ such that $t\in(t_k,t_{k+1})$, then
\begin{equation*}
\begin{aligned}
       V_g'(t,x(t))&= \frac{\partial V}{\partial_g t}(t,x(t))+\frac{\partial V}{\partial x}(t,x(t)) f(t,x(t))\\
 &=-\frac{2x(t)}{\prod_{i=1}^{k}(1+\nu_i)^2}t\arctan(x(t))\\
      &\le -\frac{2t}{\prod_{i=1}^{k}(1+\nu_i)^2} \frac{x(t)^2}{1+x(t)^2}.
  \end{aligned}
\end{equation*}
For $t_k\in [t_0,\infty)\cap D_g$, we have that
\begin{equation*}
\begin{aligned}
       V_g'(t_k,x(t_k))&=\frac{V(t_k^+,x(t_k^+))-V(t_k,x(t_k))}{g(t_k^+)-g(t_k)}\\
&=\frac{V(t_k^+,x(t_k)+ \mu_g(\{t_k\})f(t_k,x(t_k)))-V(t_k,x(t_k))}{g(t_k^+)-g(t_k)}\\
       &=\frac{\frac{(1+\nu_k)^2x(t_k)^2}{\prod_{i=1}^{k}(1+\nu_i)^2}- \frac{x(t_k)^2}{\prod_{i=1}^{k-1}(1+\nu_i)^2} }{g(t_k^+)-g(t_k)}\\
 &= 0 .
  \end{aligned}
\end{equation*}
 Therefore, we conclude that
$$
 V_g'(t,x(t))\le-\omega(t)\phi(|x(t)|), \text{ for$g$-almost all } t\ge t_0,
 $$
where $w:\mathbb{R}^+ \to\mathbb{R}^+$ is the function defined for every $t\in\mathbb{R}^+$ by
$$
w(t)=\begin{dcases}
      2t & \mbox{if } t\in[0,t_1), \\
       0 & \mbox{if }t=t_k,\, k\in \mathbb{N},\\
      \frac{2t}{\prod_{i=1}^{k}(1+\nu_i)^2}  & \mbox{if } t\in(t_k,t_{k+1}),\, k\in \mathbb{N},
     \end{dcases}
$$
and $\phi\in \mathcal{K}$ the function given by $\phi(y)=\frac{y^2}{1+y^2}$ for all $y\in\mathbb{R}^+$.

Observe that the function $w$ satisfies
$$
\begin{aligned}
\lim\limits_{t\to +\infty} \inf_{t_0 \in \mathbb{R}^+}  \int_{[t_0,t_0+t)}w(s)\,d\mu_g(s)&\ge \lim\limits_{t\to +\infty} \inf_{t_0 \in \mathbb{R}^+}\int_{t_0}^{t_0+t}a_02s\,ds\\
&=a_0\lim\limits_{t\to +\infty} \inf_{t_0 \in \mathbb{R}^+} 2tt_0+t^2\\
&=+\infty.
\end{aligned}
$$
By means of Theorem~\ref{thm:uniform asymptotic stability}, we deduce that $x=0$ is uniformly asymptotically stable equilibrium. Figure~\ref{fig:exm:asym-stability-3} illustrates the asymptotic behavior of solutions of the dynamical system~\eqref{eq:exm:asym-stability-3} with $D_g=\mathbb{N}$.
\begin{figure}[H]
  \centering
  \includegraphics[height=7.5cm]{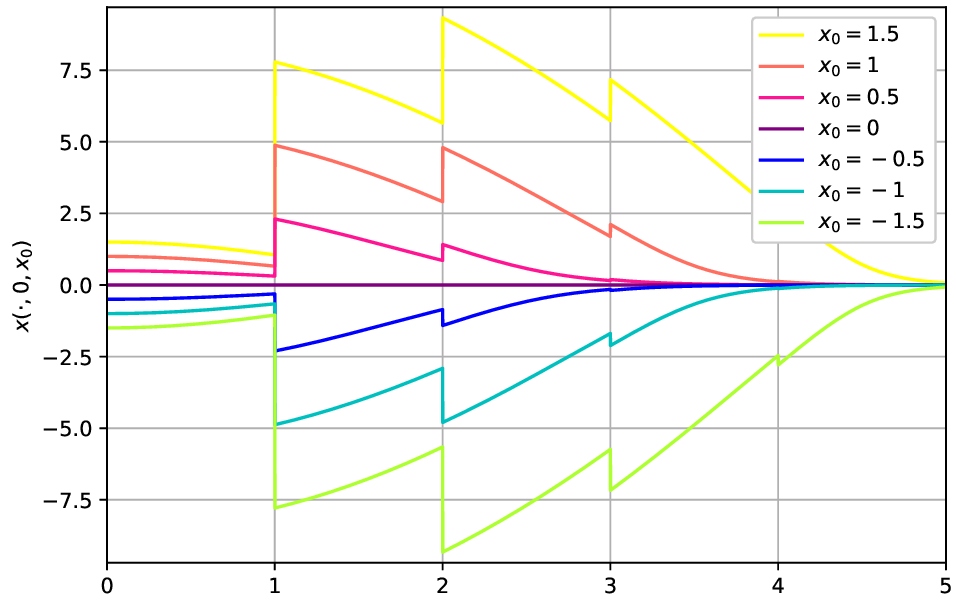}
  \caption{Asymptotic behavior of solutions of the dynamical system~\eqref{eq:exm:asym-stability-3} with $D_g=\mathbb{N}$, obtained with a time-discretization step-size of $10^{-3}$, with $\nu_k=e^{\frac{2}{k^2}}-1$ for all $k\in \mathbb{N}$.}\label{fig:exm:asym-stability-3}
\end{figure}
\end{exm}
\begin{rem}
  Observe in Example~\ref{exm:asym-stability-1} that the function $w$ satisfies
$$
\lim_{t\to +\infty} \inf_{t_0 \in \mathbb{R}^+}\int_{[t_0,t_0+t)} w(s)\,d\mu_g(s) =0 \neq \infty.
$$
Thus, Condition~{{\rm(d)}} of Theorem~\ref{thm:uniform asymptotic stability} does not hold. Consequently, uniform asymptotic stability cannot be deduced.
\end{rem}

In the classical case where $g\equiv \operatorname{id}_{\mathbb{R}}$, corollary results~\cite[Theorem~4.2]{H1} and~\cite{L1} are well-known when Conditions~{{\rm(b)}} of Theorem~\ref{thm:asymptotic stability} is replaced by $V_g'(t,\mathbf{x}(t))
$ being negative definite along each maximal solution $\mathbf{x}$ for every $t\ge t_0$. However, to present an analogous statement, we require an additional assumption to avoid the case when  $\lim\limits_{t\to\infty} g(t)=l<\infty$, and in particular,  when there exists $T\ge0$ such that $(T,\infty)\subset C_g$, Example~\ref{exm:changeof-stability-properties-linear-Stieltjes-dyn-system} provides an interesting illustration of attractivity lack for asymptotic stability.
\begin{cor}\label{cor:asymptotic stability:V<0}
Assume that  $\lim\limits_{t\to\infty} g(t)=\infty$.  If there  exist $V:\mathbb{R}^+\times B_{\mathbb{R}^n}(\mathbf{0},r_0) \to \mathbb{R}$; $V\in \mathcal{V}_1^g$ and $a,b,\phi\in \mathcal{K}$ such that
 \begin{enumerate}
\item[{{\rm(a)}}] $a(\|\mathbf{u}\|) \le V(t,\mathbf{u})\le b(\|\mathbf{u}\|)$, for every $(t,\mathbf{u})\in\mathbb{R}^+\times B_{\mathbb{R}^n}(\mathbf{0},r_0)$;
\item[{{\rm(b)}}]  for every $(t_0,\mathbf{x}_0) \in \mathbb{R}^+\times B_{\mathbb{R}^n}(\mathbf{0},r)$, the maximal solution $\mathbf{x}:I_{t_0,\mathbf{x}_0}\to B_{\mathbb{R}^n}(\mathbf{0},r_0)$ of the system~\eqref{eq:stability-dyn-system} satisfies
$$
V_g'(t,\mathbf{x}(t))\le -\phi(\|\mathbf{x}(t)\|), \text{ for $g$-almost all } t\ge t_0.
$$
     \end{enumerate}
then, the trivial solution $\mathbf{x}=\mathbf{0}$ of the system~\eqref{eq:stability-dyn-system} is  asymptotically stable. Furthermore, if $\lim\limits_{t\to +\infty}\inf_{t_0 \in \mathbb{R}^+} g(t+t_0)-g(t_0)=\infty$, then, the trivial solution $\mathbf{x}=\mathbf{0}$ of the system~\eqref{eq:stability-dyn-system} is uniformly asymptotically stable.
\end{cor}
\begin{proof}  Observe that Conditions of Theorem~\ref{thm:asymptotic stability} hold for $\omega\equiv 1$ as
$$
\inf_{t_0 \in \mathbb{R}^+} \lim\limits_{t\to +\infty}\int_{[t_0,t+t_0)}\omega(s)\,d\mu_g(s) = \inf_{t_0 \in \mathbb{R}^+}\lim\limits_{t\to +\infty}g(t+t_0)-g(t_0)=+\infty.
$$
Hence, the trivial solution is $\mathbf{x}=\mathbf{0}$ of the system~\eqref{eq:stability-dyn-system} is  asymptotically stable. Moreover, if $\lim\limits_{t\to +\infty}\inf_{t_0 \in \mathbb{R}^+} g(t+t_0)-g(t_0)=\infty$, then, Theorem~\ref{thm:uniform asymptotic stability} ensures that $\mathbf{x}=\mathbf{0}$  is uniformly asymptotically stable.
\end{proof}
\section{Applications to dynamics of population}
\subsection{Stable equilibrium of a population subject to train vibrations}
In this subsection, by means of a system of Stieltjes differential equations, we study the long-term impact of high-speed train vibrations and  noise pollution on a population of animals living near railways. Depending on the species and their sensitivity to vibrations, various implications can be observed, we cite for instance:
\begin{itemize}
  \item[$\bullet$] Hearing damage resulting in from the significant noise and vibrations that can potentially harm animals with sensitive hearing such as certain small mammals, birds, and bats which rely heavily on their hearing for communication, navigation, detection of predators, and finding food, thus, prolonged exposure to train vibrations may lead to hearing impairment or damage, disrupting their normal behaviors and increasing their death rate.
  \item[$\bullet$] Increased stress levels since some animals may be startled by the vibrations. This can affect their feeding patterns, reduce their reproduction rate, or lead to emigration resulting in a loss of suitable habitat and altering the composition and diversity of the local ecosystem
  \item[$\bullet$] Ecological interactions disruption which can affect pollination  for instance if the vibrations deter insects that are important pollinators, disturb ground-dwelling organisms (insects, reptiles, and small mammals\dots) which can implicitly impact other species that rely on them as a food source.
\end{itemize}
To these aims, an Allee effect can be observed in this regard, especially when the survival of the population depends on a minimum threshold size $M>0$. In the follows, we denote $x(t)$ as the number of individuals of a population living in a region near a railway with a carrying capacity $K>M$.
Let us assume that a certain number $m>0$ of trains pass through the area every day. We refer to $\{\tau_i\}_{i=1}^{i=m}$ as the moments when trains pass in a single day.
Once a train pass by the area, its impact is significant for a proportion of individuals living near the railway. They may experience vibrations or direct injuries. In the following analysis, we use a Stieltjes differential equation to model the dynamics of this population affected by train vibrations, and we study the asymptotic behavior of its solutions. In doing so,  we require a derivator $g:\mathbb{R}\to \mathbb{R}$ presenting discontinuities for $t\in\{\tau_i\}_{i=1}^{i=m}+24\mathbb{Z}$ such that $\mu_g(\{t\})$  quantifies the rate at which the risk of damage varies. Depending on the specific $\tau_i$, this rate can either increase or decrease, reflecting the varying impact of vibrations during daylight and nighttime hours. For simplicity,  we can take for instance:
$$
        g(t)=t+\sum_{k\in\mathbb{N}}\sum_{i=1}^{m}\chi_{[\tau_i+24k,\infty)}(t), \quad \text{for all } t\in \mathbb{R},
$$
with $\mu_g(\{\cdot\})\equiv1$ on $D_g= \{\tau_i\}_{i=1}^{i=m}+24\mathbb{Z}$.

In the sequel, we suggest to analyse the asymptotic behaviour of the dynamics of this population, through the study of asymptotic stability of the zero equilibrium of the Stieltjes dynamical system:
\begin{equation}\label{eq:model-1}
     x_g'(t) =f(t,x(t)),\text{ for $g$-almost every $t\ge t_0\ge0$,} \qquad x(t_0) =x_0,
\end{equation}
where $f: \mathbb{R}^+\times\mathbb{R}\to \mathbb{R}$ is defined by
\begin{equation}\label{eq:f-model1}
f(t,x) =\begin{dcases}
                  \rho x\Big(1-\frac{x}{K}\Big)\Big(\frac{x}{M}-1\Big), & \mbox{if } t\notin\{\tau_i+24k\}_{i=1}^{m},\,\,k=0,1,2\dots \\
                  -dx, & \mbox{if } t\in\{\tau_i+24k\}_{i=1}^{m},\,\,k=0,1,2\dots
                  \end{dcases}
\end{equation}
The parameters of the model can be understood as:
\begin{description}
\item[$K>M$] the carrying capacity of the environment.
\item[$\rho>0$] intrinsic rate of reproduction of the population.
\item[$d\in(0,1)$] Constant related to the impact induced by trains, either a migration rate immediately following the passage of trains or a mortality rate for certain populations that live in close proximity to the railway.
\end{description}

To simplify the analysis, we make the assumption that a train passes every hour over a 24-hour period. Thus, $g(t)= t+\sum_{i\in \mathbb{Z}}\chi_{[i+1,\infty)}(t)$ for all $t\in \mathbb{R}$, $D_g=\mathbb{Z}$  and $\mu_g(\{t\})=1$ for all $t\in D_g$.

The dynamics present several equilibria. However, we will focus on the zero equilibrium, to study its local asymptotic stability within a region $(-r_0,r_0)$; $r_0>0$, that will be determined to enhance the impact of environmental factors threatening this population.

Since the trivial solution $x=0$ is an equilibrium of the dynamical system~\eqref{eq:model-1},  let us consider $r_0\le M$. For all $(t_0,x_0) \in(\mathbb{R}^+\cap D_g)\times (-r_0,r_0)$, if $x$ is a solution of~\eqref{eq:model-1}, then $x(t_0^+)=x(t_0)+\mu_g(\{t_0\})f(t_0,x(t_0))=(1-d)x(t_0) \in(-r_0,r_0)$. Thus, $f$ satisfies~{\rm(H$_{r}$)} for $r=r_0$. Combining this with Theorem~\ref{thm:existence of maximal solution}, yields existence of a maximal solution $x: I_{t_0,x_0}\to \mathbb{R}$ for every $(t_0,x_0) \in \mathbb{R}^+\times(-r_0,r_0)$. Let $\omega(t_0,x_0):=\sup I_{t_0,x_0}$. Let us construct this solution through local existence to show that $\omega(t_0,x_0)=\infty$. First of all, let $\tau\in(t_0,\omega(t_0,x_0))$ such that $x:[t_0,\tau]\to (-r_0,r_0)$ is a solution of~\eqref{eq:model-1}. Observe that if $x(\tau)=0$, by uniqueness of the solution, we deduce that $x\equiv 0$ which lies in $(-r_0,r_0)$. Thus, two other interesting cases occur when $x(\tau) \neq 0$:

\textbf{Case 1:} if $\tau \in D_g$, then $x(\tau^+)=x(\tau)+\mu_g(\{\tau\})f(\tau,x(\tau))=(1-d)x(\tau)\in (-r_0,r_0)$.

\textbf{Case 2:} if $\tau\notin D_g$, then we distinguish two subcases:

\textbf{\quad Subcase 1:} if $x(\tau)\in(0,r_0)$ then
$x_g'(\tau)=f(\tau,x(\tau))<0$. Thus, by $g$-continuity of $x$ at $\tau$ there exists $\tau_1 \in(\tau,\omega(t_0,x_0))$ such that $x:[t_0,\tau_1]\to (0,r_0)$.

\textbf{\quad Subcase 2:} if $x(\tau)\in(-r_0,0)$ then $x_g'(\tau)=f(\tau,x(\tau))>0$. Thus, there exists  $\tau_2 >0$ such that $x:[\tau,\tau_2]\to (-r_0,0)$.

Repeating the same argument for each subinterval of $I_{t_0,x_0}$, we deduce that the solution $x(t)\in[-\lambda_{x_0},\lambda_{x_0}]\subset(-r_0,r_0)$ for all $t\in I_{t_0,x_0}$, where
$$
\lambda_{x_0}=\sup_{t\in[t_0,\tau]}|x(t,t_0,x_0)|.
$$
By Corollary~\ref{cor:compacity yields tm=infty}, we deduce that $\omega(t_0,x_0)=\infty$. Hence, $x:[t_0,\infty)\to (-r_0,r_0)$ is a maximal solution of~\eqref{eq:model-1}.

Now, we define the function $V:\mathbb{R}^+\times (-r_0,r_0)\to \mathbb{R}$ by $V(t,x)=x^2$ for all $(t,x) \in\mathbb{R}^+\times(-r_0,r_0)$. Clearly $V\in \mathcal{V}_g^1$. For all $(t_0,x_0) \in\mathbb{R}^+\times (-r_0,r_0)$, let $x:[t_0,\infty)\to (-r_0,r_0)$ be the maximal solution of~\eqref{eq:model-1}. Thus, using Proposition~\ref{prop:total g-derivative}, we obtain for $g$-almost all  $t \in [t_0,\infty)\setminus D_g$:
\begin{equation*}
  \begin{aligned}
       V_g'(t,x(t))&=
 \frac{\partial V}{\partial_g t}(t,x(t))+\frac{\partial V}{\partial x}(t,x(t)) x_g'(t)\\
 &=\frac{\partial V}{\partial_g t}(t,x(t))+\frac{\partial V}{\partial x}(t,x(t)) f(t,x(t))\\
       &= 2\rho x(t)^2\Big(1-\frac{x(t)}{K}\Big)\Big(\frac{x(t)}{M}-1\Big).
  \end{aligned}
\end{equation*}
While for $t \in [t_0,\infty)\cap D_g$, we obtain
\begin{equation*}
  \begin{aligned}
       V_g'(t,x(t))&=\frac{V(t^+,x(t^+))-V(t,x(t))}{g(t^+)-g(t)}\\
       &=\frac{V(t^+,x(t)+ \mu_g(\{t\})f(t,x(t)))-V(t,x(t))}{g(t^+)-g(t)}\\
       &=(1-d)^2x(t)^2-x(t)^2\\
       &=(-2d+d^2)x(t)^2.
  \end{aligned}
\end{equation*}
As $(-2d+d^2)<0$, it follows that $V_g'(t,x(t))$ is negative definite. Since $a(|x|)\le V(t,x)\le b(|x|)$ for all $(t,x)\in \mathbb{R}^+\times(-r_0,r_0)$ with $a,b\in\mathcal{K}$ defined by $a(s)=b(s)=s^2$ for all $y\in\mathbb{R}^+$,  Corollary~\ref{cor:asymptotic stability:V<0} ensures that the trivial solution $x=0$ is  uniformly  asymptotically stable. The asymptotic behavior of solutions is illustrated in Figure~\ref{fig:Asymp-Behav-model1}.
\subsection*{Comments}
Figure~\ref{fig:Asymp-Behav-model1} illustrates the long-term effect of the high-speed trains vibrations. As shown in the figure,  if the vibrations lead to an Allee effect, the population will decline significantly over time.  This outcome raises alarm about the overall stability and the persistence of the whole ecosystem in this region. Specifically, it indicates a high likelihood of population extinction  when there is no estimation showing that the initial population exceeds $M$.
\begin{figure}[H]
  \centering
  \includegraphics[width=12cm]{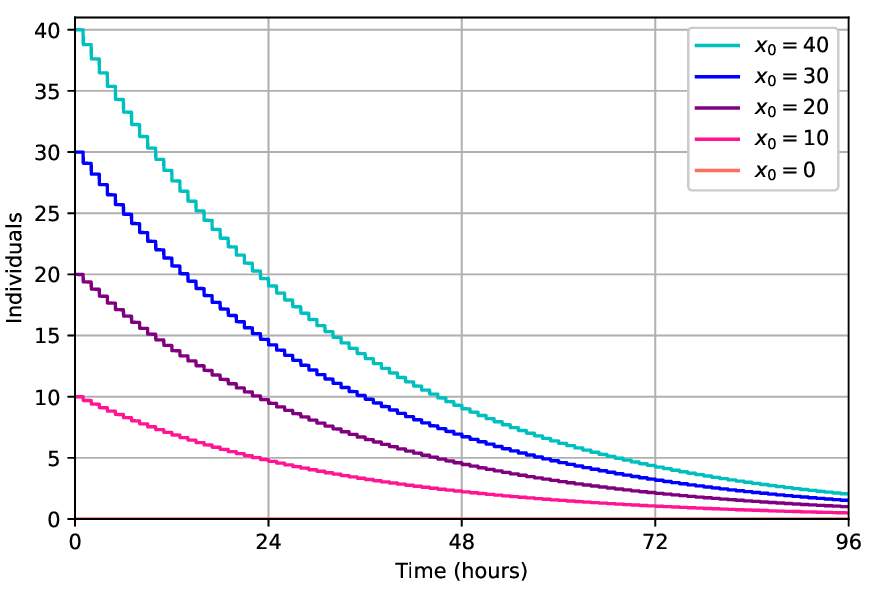}
  \caption{Asymptotic behavior of solutions of the dynamical system~\eqref{eq:model-1} obtained with a time-discretization step-size of $10^{-3}$, with $\rho=0.001$, $K=100$, $M=50$, and $d=0.03$.}\label{fig:Asymp-Behav-model1}
\end{figure}

\subsection{Stable equilibrium of Bacteria-Ammonia dynamics}
{\it Cyanobacteria}, similar to plants, participate in oxygenic photosynthesis as they are photosynthetic bacteria. {\it Photosynthesis } is the biochemical process through which organisms convert light energy into chemical energy in the form of glucose or other organic compounds resulting oxygen release. The energy captured is then used to fuel the synthesis of organic molecules such as glucose, serving as an energy source for the bacteria.

Cyanobacteria  are found in diverse habitats, including freshwater, marine environments, and terrestrial ecosystems. In this section, we consider a species of cyanobacteria that has also the ability to fix nitrogen such as {\it Anabaena cyanobacteria}. Through nitrogen fixation, atmospheric nitrogen gas $(N_2)$ is converted into a form that can be used by plants and other organisms by means the enzyme {\it nitrogenase}. This enzyme catalyzes the conversion of nitrogen gas $(N_2)$ into ammonium ions $(NH_4^+)$ based on a considerable amount of energy obtained from photosynthesis, these ammonium ions are assimilated into amino acids and proteins, which are essential for the growth and survival of Cyanobacteria. As a by product of nitrogen fixation, ammonia gas $(NH_3)$ can be  released. Some of the assimilated ammonium ions $(NH_4^+)$ are released back into the environment providing neighboring vegetative cells with a source of nitrogen.

To avoid losing valuable nitrogen nutrients and optimizing nitrogen utilization efficiency, this population has mechanisms allowing ammonium ions $(NH_4^+)$ reabsorption, and ammonia gas $(NH_3)$ assimilation through converting ammonia gas $(NH_3)$ into ammonium ions $(NH_4^+)$. In what follows, the term "ammonia" refers to both the protonated and unprotonated forms, which are denoted as $(NH_3)$ and $(NH_4^+)$ respectively~\cite{B}. It is worth mentioning that ammonia is commonly used in cleaning products, fertilizers. Beyond that, ammonia's cooling properties make it an essential refrigerant in air conditioning systems and refrigerators.

To optimize resource utilization and adapt to the varying environmental conditions, the population undergoes a day-night cycling of nitrogen fixation and carbon consumption~\cite{WRSHSG}. This is due to the sensitivity of the nitrogenase enzyme responsible for nitrogen fixation to oxygen. Thus, nitrogen fixation is carried out during daylight hours when the photosynthesis can provide the necessary energy and oxygen levels are relatively low. During the nights, since oxygen levels within the cells would be higher, this population of cyanobacteria reduces their metabolic activity and growth and relies on stored carbon compounds to fulfill their energy needs. Our objective in the sequel is to observe the dynamics of this population which thrives in the presence of ammonia in a culture room without exposure to artificial light,  tracking the levels of the ammonia during this the process. Here, we assume that  the carbon dioxide $(CO_2)$, the nitrogen gas $(N_2)$ and nutrients supply are maintained steady as well as the PH level.  Since the population undergoes dormancy phases during the nights, we identify the days with intervals of the form $[2k,2k+1]$, $k=0,1,2,\dots$, and the night with intervals $[2k+1,2k+2]$, $k=0,1,2,\dots$ In our example, we differentiate with respect to a derivator $g$ whose variation describes the intensity of light, which is necessary for the photosynthesis process. We require that $g$ presents smaller slops at the beginning and at the end of the daylight hours, with maximal slops at middays where $t=2k+1/2$, and remains constant during the dormancy phases in the night $[2k+1,2k+2]$, $k=0,1,2,\dots$ For instance, we consider $g:\mathbb{R} \to \mathbb{R}$ defined by
$$
g(t)=\begin{dcases}
      \frac{\sin(\pi(t-1/2))+1}{2} & \mbox{if } t\in[0,1] \\
       1 &\mbox{if } t\in(1,2],
     \end{dcases}
$$
and $g(t)=g(1)+g(t-2)$ for $t\ge 2$.
\begin{figure}[H]
  \centering
  \includegraphics[height=7.5cm]{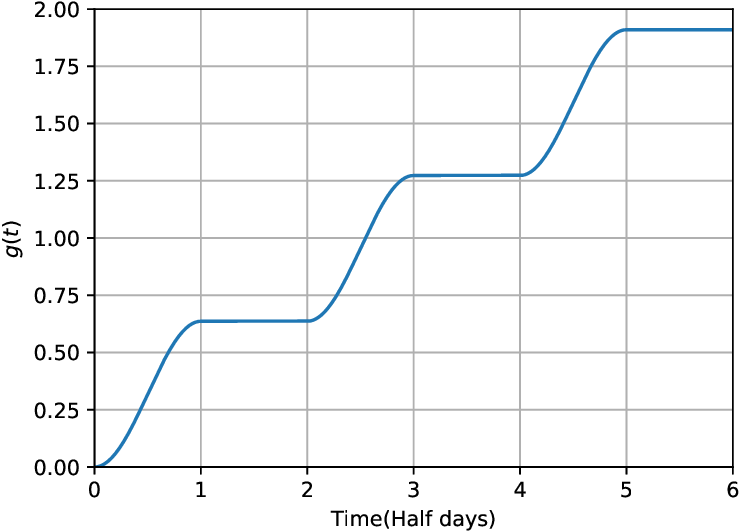}
  \caption{Graph of the derivator $g$.}
\end{figure}
We denote $N(t)$ the Biomass of cyanobacteria~(grams per liter), and $A(t)$ the ammonia concentration in the environment at time $t\ge 0$.  Since the population thrives in the presence of ammonia, we can assume that the growth rate is proportional to the level of ammonia. We denote $\rho$ the maximal intrinsic coefficient of reproduction that the population can reach in the presence of one unit of ammonia with maximal sunlight intensity. Thus, the dynamics can be modeled  using the autonomous system of Stieltjes differential equations:
\begin{equation}\label{eq:model-2}
\begin{aligned}
     &\mathbf{u}_g'(t) = \mathbf{F}(\mathbf{u}(t)),\text{ for $g$-almost every $t\ge t_0\ge0$,}
     \\
      &\mathbf{u}(0) =\mathbf{u}_0 =(N_0,A_0),
      \end{aligned}
\end{equation}
where $\mathbf{u}=(N,A)$ and $\mathbf{F}=(F_1,F_2) : \mathbb{R}^2\to \mathbb{R}^2$ is defined by
\begin{equation}\label{eq:F-model2}
\begin{aligned}
F_1(N,A)  &=  \rho AN\Big(1-\frac{N}{K}\Big)\\
F_2(N,A)  &= \alpha N-\beta AN.
\end{aligned}
\end{equation}
where the parameters of the model can be understood as:
\begin{description}
\item[$K>0$] the carrying capacity of the culture room, which forms a spacial constraint for growth.
\item[$\alpha>0$] Constant related to the production of ammonia through nitrogen fixation.
\item[$\beta>0$] Constant related to the proportion of the reabsorption of ammonia by the population depending on the level of ammonia in the environment.
\end{description}
Observe that $\mathbf{u}^*=(K,\alpha/\beta)$ is an equilibrium of the system~\eqref{eq:model-2} among other equilibria. Its asymptotic stability would guarantee the persistence of the population with nonzero ammonia production. Therefore, we study local asymptotic stability in a neighborhood $B_{\mathbb{R}^2}(\mathbf{u}^*,r_0)$; $r_0>0$, of this equilibrium. To this aim, we transfer our study in a neighborhood of $\mathbf{x}=\mathbf{0}=(0,0) \in \mathbb{R}^2$. If we adopt the variable change $\mathbf{x}=(x_1,x_2):=\mathbf{u}-\mathbf{u}^*$, we obtain the system
\begin{equation}\label{eq:centered model2}
\begin{aligned}
     &\mathbf{x}_g'(t) = \mathbf{f}(\mathbf{x}(t)),\text{ for $g$-almost every $t\ge t_0\ge0$,}
     \\
      &\mathbf{x}(0) =\mathbf{x}_0 :=\mathbf{u}_0-\mathbf{u}^*,
      \end{aligned}
\end{equation}
where $\mathbf{f}=(f_1,f_2) : \mathbb{R}^2\to \mathbb{R}^2$ is defined by
\begin{equation}\label{eq:f-model}
\begin{aligned}
f_1(x_1,x_2)  &=  -\rho\frac{x_1}{K}(x_2+\alpha/\beta)(x_1+K)\\
f_2(x_1,x_2)  &=  -\beta x_2(x_1+K).
\end{aligned}
\end{equation}
$\mathbf{x}=\mathbf{0}=(0,0) \in \mathbb{R}^2$ is an equilibrium of the dynamical system~\eqref{eq:centered model2}. Next, we prove that $\mathbf{x}=\mathbf{0}$ is  asymptotically stable, implying that $\mathbf{u}^*=(K,\alpha/\beta)$ is an asymptotically stable equilibrium of the system~\eqref{eq:model-2}.

Let us consider $r_0=\min\{K,\alpha/\beta\}$, and let $(t_0,\mathbf{x}_0) \in \mathbb{R}^+\times B_{\mathbb{R}^2}(\mathbf{0},r_0)$. Arguing as in the previous subsection, since the function $\mathbf{f}$ satisfies conditions of Theorem~\ref{thm:existence of maximal solution}, we prove existence of the maximal solution $\mathbf{x}=(x_1,x_2):[t_0,\infty)\to B_{\mathbb{R}^2}(\mathbf{0},r_0)$ of the system~\eqref{eq:centered model2}. Initially, let $\mathbf{x}: I_{t_0,\mathbf{x}_0}\to \mathbb{R}^2$  be the maximal solution of~\eqref{eq:centered model2} with $\omega(t_0,\mathbf{x}_0):=\sup I_{t_0,\mathbf{x}_0}$. Let $\tau\in(0,\omega(t_0,\mathbf{x}_0))\setminus C_g$ such that $\mathbf{x}=(x_1,x_2):[t_0,\tau]\to B_{\mathbb{R}^2}(\mathbf{0},r_0)$ is a solution of~\eqref{eq:model-1}. In what follows, we analyze the possible cases:

\textbf{Case 1:} if $x_1(\tau),x_2(\tau)\in(0,r_0)$ (resp. $x_1(\tau),x_2(\tau)\in(-r_0,0)$),  then, for $i=1,2$, we have
$(x_i)_g'(\tau)=f_i(\mathbf{x}(\tau))<0$ (resp. $(x_i)_g'(\tau)=f_i(\mathbf{x}(\tau))>0$). Therefore, by $g$-continuity of $x$ at $\tau$, there exists $\tau_1 \in(\tau,\omega(t_0,\mathbf{x}_0))$ ($\tau_1$ can be chosen such that $\tau_1\notin C_g$) such that $x_i:[t_0,\tau_1]\to (0,r_0)$ (resp. $x_i:[t_0,\tau_1]\to (-r_0,0)$).

\textbf{Case 2:} if $x_1(\tau)\in(0,r_0)$ and $x_2(\tau)\in(-r_0,0)$, then
$(x_1)_g'(\tau)=f_1(\mathbf{x}(\tau))<0$ and $(x_2)_g'(\tau)=f_2(\mathbf{x}(\tau))>0$.  Therefore, there exists $\tau_2 \in(\tau,\omega(t_0,\mathbf{x}_0))\setminus C_g$ such that $x_1:[\tau,\tau_2]\to (0,r_0)$, and $x_2:[\tau,\tau_2]\to (-r_0,0)$.

\textbf{Case 3:} if $x_1(\tau)\in(-r_0,0)$ and $x_2(\tau)\in(0,r_0)$, then similarly to \textbf{Case 2}, we deduce that there exists $\tau_3 \in(\tau,\omega(t_0,\mathbf{x}_0))$ such that $x_1:[\tau,\tau_3]\to(-r_0,0)$, and $x_2:[\tau,\tau_3]\to (0,r_0)$.

Repeating the same argument for each subinterval of $I_{t_0,\mathbf{x}_0}$, we deduce that the solution $\mathbf{x}(t)\in [-\lambda_{\mathbf{x}_{0,1}},\lambda_{\mathbf{x}_{0,1}}]\times [-\lambda_{\mathbf{x}_{0,2}},\lambda_{\mathbf{x}_{0,2}}]\subset B_{\mathbb{R}^2}(\mathbf{0},r_0)$ for all $t\in I_{t_0,\mathbf{x}_0}$, where
$$
\lambda_{\mathbf{x}_{0,i}}=\sup_{t\in[t_0,\tau]} |x_i(t,t_0,\mathbf{x}_0)|, \quad\text{ for } i=1,2\text{ and }\mathbf{x}_0=(\mathbf{x}_{0,1} \mathbf{x}_{0,2}).
$$
Using Corollary~\ref{cor:compacity yields tm=infty}, we obtain that $\omega(t_0,\mathbf{x}_0)=\infty$. Hence, $x:[t_0,\infty)\to B_{\mathbb{R}^2}(\mathbf{0},r_0)$ is the maximal solution of~\eqref{eq:centered model2}.

Now, let us consider the function $V:\mathbb{R}^+\times B_{\mathbb{R}^2}(\mathbf{0},r_0) \to \mathbb{R}$ defined by $V(t,\mathbf{x})=x_1^2+x_2^2$ for all $t\ge 0$ and $\mathbf{x}=(x_1,x_2)\in \mathbb{R}^2$. It is clear that $V\in \mathcal{V}_1^g$.
Let $(t_0,\mathbf{x}_0)\in \mathbb{R}^+\times B_{\mathbb{R}^2}(\mathbf{0},r_0)$, and $\mathbf{x}:[t_0,\infty)\to B_{\mathbb{R}^2}(\mathbf{0},r_0)$ be the maximal solution of~\eqref{eq:centered model2}. By  means of Proposition~\ref{prop:total g-derivative}, for $g$-almost all  $t \in [t_0,\infty)$, we obtain
\begin{equation}\label{eq:total g-derivative of centered model2}
\begin{aligned}
V_g'(t,\mathbf{x}(t))&=
 \frac{\partial V}{\partial_g t}(t,\mathbf{x}(t))+ \sum_{i=1}^{2}\frac{\partial V}{\partial x_i}(t,\mathbf{x}(t)) (x_i)_g'(t)\\
 &=\frac{\partial V}{\partial_g t}(t,\mathbf{x}(t))+ \sum_{i=1}^{2}\frac{\partial V}{\partial x_i}(t,\mathbf{x}(t)) f_i(t,\mathbf{x}(t))\\
 &= -2\rho\frac{x_1(t)^2}{K}(x_2(t)+\alpha/\beta)(x_1(t)+K)-2\beta x_2(t)^2(x_1(t)+K).
 \end{aligned}
\end{equation}
Thus, $V_g'(t,\mathbf{x}(t))$ is negative definite. Since $a(\|\mathbf{z}\|)\le V(t,\mathbf{z})\le b(\|\mathbf{z}\|)$ for all $(t,\mathbf{z})\in \mathbb{R}^+\times B_{\mathbb{R}^2}(\mathbf{0},r_0)$ where $a,b\in\mathcal{K}$ are defined by $a(y)=y^2$ and $b(y)=2y^2$ for all $y\in\mathbb{R}^+$, it follows from Corollary~\ref{cor:asymptotic stability:V<0} that $\mathbf{x}=\mathbf{0}$ is  uniformly asymptotically stable. Hence, $\mathbf{u}^*=(K,\alpha/\beta)$ is a uniformly asymptotically stable equilibrium of the system~\eqref{eq:model-2}. The graph of asymptotic behavior of solutions near the equilibrium $\mathbf{u}^*=(K,\alpha/\beta)$ is given in Figure~\ref{fig:Asymp-Behav-model2}.
\begin{figure}[H]
 \centering
  \includegraphics[width=13cm]{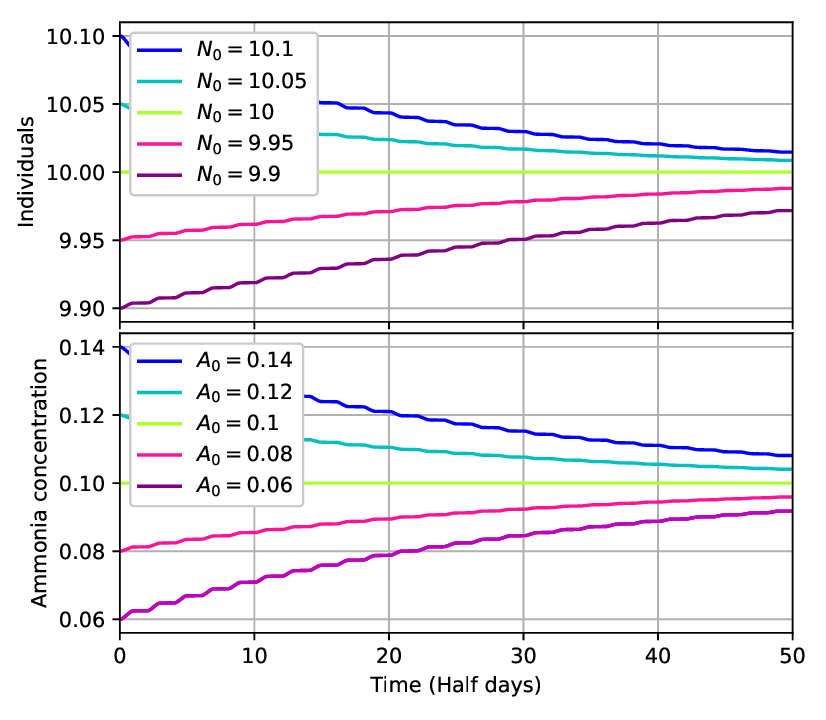}
  \caption{Asymptotic behavior of solutions  of the dynamical system~\eqref{eq:model-2} obtained with a time-discretization step-size of $10^{-3}$, with different initial densities $N_0$ and ammonia concentrations $A_0$ (g/L), where   $\rho=1$, $K=10$, $\alpha=0.001$, and $\beta=0.01$.}\label{fig:Asymp-Behav-model2}
\end{figure}

\section*{Funding}
Lamiae Maia was partially supported by the National Center of Scientific and Technical Research (CNRST) under Grant No. 60UM5R2021, Morocco.
\section{Acknowledgment}
Lamiae Maia would like to express her sincere gratitude towards Professor Marlène Frigon and the {\it Département de mathématiques et de statistique} of the {\it Université de Montréal}, for their warm hospitality and for funding her research stay at the aforementioned department when this article was finalized.

\end{document}